\documentclass[a4paper]{amsart}
\allowdisplaybreaks
\usepackage[utf8]{inputenc}
\usepackage[T1]{fontenc}
\usepackage{lmodern}

\usepackage[svgnames]{xcolor}

\usepackage[english]{babel}

\usepackage{csquotes}

\usepackage[backend=biber,
            sorting=nyt,
            sortcites=true,
            maxbibnames=9,
            maxalphanames=4,
            giveninits=true,
            url=true,
            isbn=false]{biblatex}
\DeclareFieldFormat{doi}{%
  \mkbibacro{DOI}\addcolon\space
  \ifhyperref
    {\href{https://doi.org/#1}{\MakeLowercase{\nolinkurl{#1}}}}
    {\MakeLowercase{\nolinkurl{#1}}}
}
\DeclareSourcemap{
  \maps{
    \map{
      \pertype{online}
      \step[fieldset=url, null]
      \step[fieldset=urldate, null]
      \step[fieldset=doi, null]
      \step[fieldset=pubstate, null]
    }
    \map{
      \pertype{article}
      \step[fieldset=url, null]
      \step[fieldset=urldate, null]
    }
    \map{
      \pertype{incollection}
      \step[fieldset=url, null]
      \step[fieldset=urldate, null]
    }
    \map{
      \pertype{book}
      \step[fieldset=pagetotal, null]
      \step[fieldset=url, null]
    }
  }
}
\usepackage{amssymb}
\usepackage[colorlinks,linkcolor=Maroon,citecolor=Teal,urlcolor=Navy]{hyperref}

\usepackage{enumitem}
\usepackage{tensor}

\usepackage[final]{microtype}

\newcommand{\curlT}{\mathcal T}

\newcommand{\si}{\sigma}
\newcommand{\downto}{\searrow}

\newcommand{\ga}{\gamma}

\newcommand{\ep}{\epsilon}

\newcommand{\al}{\alpha}
\newcommand{\La}{\Lambda}

\newcommand{\grad}{{{}^h \nabla}}
\newcommand{\gradh}{{{}^h \nabla}}
\newcommand{\gradell}{{{}^{\ell} \nabla}}
\newcommand{\gradgt}{{{}^{g} \nabla}}
\newcommand{\gradg}{\nabla}
\newcommand{\gradgg}{{{}^{g}\nabla}}

\newcommand{\lap}{\Delta}
\newcommand{\de}{\delta}
\renewcommand{\phi}{\varphi}
\newcommand{\be}{\beta}

\newcommand{\ti}{\tilde}

\newcommand{\partt}{\frac{\partial}{\partial t}}

\newcommand{\parttau}{ \frac{\partial}{\partial \tau}}
\newcommand{\upto}{\nearrow}
\newcommand{\e}{{\rm \epsilon}}

\newcommand{\map}{{\rm \hspace{1mm} \rightarrow \hspace{1mm}}}

\newcommand{\Sc}{\mathrm{R}}
\newcommand{\Rm}{\mathrm{Rm}}
\renewcommand{\si}{\sigma}
\newcommand{\Riem}{\Rm}
\newcommand{\Ric}{\mathrm{Ric}}
\newcommand{\Rc}{\Ric}

\newcommand{\cc}{\subset\joinrel\subset}

\newcommand{\N}{\mathbb N}

\newcommand{\R}{\mathbb R}

\newtheorem{thm}{Theorem}[section]
\newtheorem{lem}[thm]{Lemma}

\theoremstyle{definition}
\newtheorem{defi}[thm]{Definition}

\theoremstyle{remark}
\newtheorem{rem}[thm]{Remark}

\numberwithin{equation}{section}

\newcommand{\abs}[1]{|#1|}

\newcommand{\der}[5]{\frac{{#4}^{#3} #1}{#5 {#2}^{#3}}}
\newcommand{\pd}[2]{\der{#1}{#2}{}{\partial}{\partial}}
 
\let\div\relax
\DeclareMathOperator{\div}{div}

\DeclareMathOperator{\vol}{vol}

\newcommand{\loc}{\mathrm{loc}}
\let\epsilon\varepsilon
\def\qqquad{\hskip3em\relax}
\def\qqqquad{\hskip4em\relax}
\def\qqqqquad{\hskip5em\relax}

\title[Ricci--DeTurck flow of almost continuous \texorpdfstring{$L^2$}{L\textasciicircum{}2}-metrics]{Ricci--DeTurck flow of almost continuous \texorpdfstring{$L^2$}{L\textasciicircum{}2}-metrics, and metrics with distributional scalar curvature bounded from below}
\date{November 28, 2025}

\author[F. Litzinger]{Florian Litzinger}
\address{Otto-von-Guericke-Universität Magdeburg, Fakultät für Mathematik, Institut für Analysis und Numerik, Universitätsplatz 2, 39106 Magdeburg, Germany}
\email{florian.litzinger@ovgu.de}

\author[M. Simon]{Miles Simon}
\address{Otto-von-Guericke-Universität Magdeburg, Fakultät für Mathematik, Institut für Analysis und Numerik, Universitätsplatz 2, 39106 Magdeburg, Germany}
\email{miles.simon@ovgu.de}

\makeatletter
\hypersetup{pdftitle=\@title,pdfauthor=\authors}
\makeatother

\begin{document}

\begin{abstract}
We consider Riemannian manifolds $(M^n,g_0)$, $(M^n,h)$, where $(M^n,h)$ is smooth, complete, with curvature bounded in absolute value by $K_0 < \infty$, and $(1-\varepsilon_0(n)) h \leq g_0 \leq (1+\varepsilon_0(n)) h$ for some small $\varepsilon_0(n)>0$.
It was shown by Simon (2002) that a Ricci--DeTurck flow solution $g(t)_{t \in (0,T)}$ related to $g_0$ exists for some $T=T(n,K_0)>0$. 
If $g_0 \in L^2_{\mathrm{loc}}$ or $g_0 \in W^{1,2+2\sigma}_{\mathrm{loc}}$, $\sigma \in (0,\frac{1}{4})$, respectively, we show that $g(t) \to g_0$ in the $L^2_{\mathrm{loc}}$- or $W^{1,2+\sigma}_{\mathrm{loc}}$-sense, respectively.
If $M$ is closed, $g_0 \in W^{1,2+\sigma}(M)$ for some $\sigma>0$, and the distributional scalar curvature of Lee--LeFloch (2015) is not less than $b \in \mathbb{R}$, then we show that $g(t)$ has scalar curvature not less than $b$ in the smooth sense for all $t>0$.
\end{abstract}

\maketitle

\section{Introduction}

In this paper we consider the Ricci--DeTurck flow (and  related Ricci flows) of Riemannian manifolds $(M,g_0)$ which are \emph{almost continuous} in the sense that
\begin{equation*}
  (1-\ep_0(n)) h \leq g_0 \leq (1+\ep_0(n)) h
\end{equation*}
for some appropriately chosen small positive $\ep_0(n)>0$, where $(M,h)$ is a smooth complete space with bounded curvature.
The existence and regularity of this flow was studied in Simon~\cite{simon2002deformation}: see Theorem~\ref{C0RicciDeTurckthm} of this paper for an overview of some of the results obtained therein. In a paper that first appeared seven years later, Koch--Lamm investigated existence and regularity for the Ricci--DeTurck flow  (and other geometric flows) of almost continuous metrics using a direct fixed point argument \cite{KochLamm1} (see also \cite{KochLamm2}). 

The first result of the present paper shows that if $g_0$ is an almost continuous metric which is locally in $L^2$, then the initial data is achieved locally in the $L^2$-sense, and if the metric $g_0$ is locally in $W^{1,2}$, then the $L^2$-initial value is locally achieved at a rate of $Ct$ (and further estimates hold).  
In the following, and in the remainder of the paper, $|\cdot|$ denotes the norm and $\gradh$ the covariant derivative, both with respect to a smooth Riemannian metric $h$. The notation $B_h(x_0,r)$ refers to a ball of radius $r>0$ centered at $x_0 \in M$, measured with respect to $h$. 

\begin{thm}
  \label{main0}
  For any $n \in \N$ there exists $\ep_0(n)>0$   such that the following holds.
  Let $(M^n,h)$ be a smooth, connected, complete $n$-dimensional Riemannian manifold such that $K_j:= \sup_M |\gradh^j \Rm(h)| < \infty$ for all $j \in \N_0$.
  Assume further that $g_0$ is a Riemannian metric on $M$ for which $(1-\ep_0(n)) h \leq g_0 \leq (1+\ep_0(n)) h$ holds.  
  \begin{enumerate}[label=(\roman*)]
    \item
    If $g_0 \in L^2_{\loc}(M)$, then the solution $g(t)_{t\in (0,T(n,K_0)]}$ 
     of the $h$-Ricci--DeTurck flow constructed in \cite{simon2002deformation} (see Theorem~\ref{C0RicciDeTurckthm} of this paper) satisfies
    \begin{equation*}
      \int_{B_h(x_0,r)} |g(t)-g_0|^2 dh + \int_{0}^t \int_{B_h(x_0,r)} |\gradh g(s)|^2 dh \, ds \to 0
    \end{equation*}
    as $t \downto 0$ for all $x_0 \in M$ and all $r>0$.
    \item
    If $g_0 \in W^{1,2}_{\loc}(M)$, then the solution  $g(t)_{t\in (0,T(n,K_0)]}$  of the $h$-Ricci--DeTurck flow constructed in \cite{simon2002deformation} (see Theorem~\ref{C0RicciDeTurckthm} of this paper) satisfies
    \begin{align*} 
      &\int_{B_h(x_0,r)} |g(t) - g_0|^2 dh + \int_0^t \int_{B_h(x_0,r)} |\gradh g(s)|^2 dh \, ds \\
      &\quad \leq t \cdot c(n, r,\ell,K_0) \left( 1 + \int_{B_h(x_0,\ell)}|\gradh g_0|^2 \right)
    \end{align*} 
    for all $x_0 \in M$, all $0<r<\ell$, and all $t \in (0,T]$.
  \end{enumerate}
\end{thm}

We then consider almost continuous metrics $g_0$ which are in $W^{1,2+2\si}_{\loc}(M,h)$ for some $\frac{1}{4}>\si \geq 0$. We show that the solution $g(t)$ constructed in Simon~\cite{simon2002deformation} converges to $g_0$ in $W^{1,2+\si}_{\loc}(M,h)$ as $t \downto 0$, and we prove estimates on the rate at which the $W^{1,2+\si}(U)$-norms can change for $U$ compactly contained in $M$.

\begin{thm}
  \label{main1}
  For any $n\in \N$ there exists $\ep_0(n)>0$ such that the following holds. 
  Let $(M^n,h)$ be a smooth, connected, complete $n$-dimensional Riemannian manifold such that $K_j:= \sup_M |\gradh^j \Rm(h)| < \infty$ for all $j \in \N_0$. 
 Assume further that $g_0$  is a Riemannian metric on $M$ for which $(1-\ep_0(n)) h \leq g_0 \leq (1+\ep_0(n)) h$ holds,  and that $g_0 \in W^{1,2+2\si}_{\loc}(B_{h}(x_0,2R))$ for some $\si \in [0,\frac{1}{4})$.

  Then the solution $g(t)_{t \in (0,T(n,K_0)]}$ of the $h$-Ricci--DeTurck flow constructed in \cite{simon2002deformation} (see Theorem~\ref{C0RicciDeTurckthm} of this paper) satisfies 
  \begin{align*} 
    &\int_{B_h(y_0,r)} |\gradh g|^{2+2\si} dh + 
    \int_0^t \int_{B_h(y_0,r)} \left( |\gradh g(s)|^{4+2\si} + |\gradh g(s)|^{2\si} |\gradh^2 g(s)|^2 \right) dh \, ds \\
    &\quad \leq c(n)\int_{B_h(y_0,\ell)} |\gradh g_0|^{2+2\si} dh + c(n,r,\ell,K_0,K_1) \cdot t,\\
    &\int_{B_{h}(x_0,R)} |g(t) - g_0|^{p} \to 0
  \end{align*}
  as $t \downto 0$
  for any $p\geq 2$, for any $y_0 \in B_{h}(x_0,R)$, $0<r<\ell<1$ with $B_h(y_0,4\ell) \subseteq B_{h}(x_0,R)$, $t\leq T$.
  If we further restrict to $\si \in (0,\frac{1}{4})$, then we also have 
  \begin{align*}
    &\int_{B_{h}(x_0,R)} |\gradh g(t)-\gradh g_0|^{2+\si} \to 0 
  \end{align*}
  as $t \downto 0$.  
\end{thm}

The last results we present here in the introduction are concerned with applications of the results mentioned above to metrics with \emph{scalar curvature bounded from below}. In the following, we denote the scalar curvature of a smooth Riemannian metric $g$ on a smooth manifold $M$ by $\Sc_{g}$. 

Recently there has been much interest in further understanding such spaces. 
See for example the books edited by Gromov--Lawson~\cite{PerSc1, PerSc2}, the surveys of Gromov~\cite{GromFour}, Sormani--Tian--Wang~\cite{tian2024oberwolfach}, Allen--Bryden--Kazaras--Perales--Sormani~\cite{allen2024}, Sormani~\cite{Sormani}, as well as the papers of Lott~\cite{LOTTSc,LOTTSc2}, Munteanu--Wang~\cite{MuntWang}, Cecchini--Hanke--Schick~\cite{CecHanSch}, Cecchini--Zeidler~\cite{CecZei}, Cecchini--Frenck--Zeidler~\cite{CecFreZei}, Bär--Hanke~\cite{BaeHan}, Lee--Tam~\cite{LeeSc}, and the references or articles in all of these books, papers, and surveys.
The list of references we have given here is merely a starting point and is by no means complete.

These spaces may arise as limits of smooth spaces with scalar curvature bounded from below by a fixed constant; or in some cases one can consider scalar curvature bounds from below of non-smooth metric spaces, if the metric space is regular enough.
For example, one can define curvature bounds from above or below of Riemannian manifolds $(M,g)$ \emph{in a distributional sense} if the metric $g$ is in $W^{1,2}_{\loc}(M,h)$ for some smooth Riemannian metric $h$ and satisfies $\frac{1}{c} h \leq g \leq c h$: see Definition~\ref{LeeLeFloch}.
In both cases the Ricci flow and the $h$-Ricci--DeTurck flow (see Section~\ref{RDeTurckRicci}) have proved a useful tool for examining and obtaining more information about such spaces. 
See for example Jiang--Sheng--Zhang~\cite{jiang2023weak}, Burkhardt-Guim~\cite{burkhardt-guim2024smoothing, Burk-Guim1, Burk-Guim2}, Bamler~\cite{bamlerricciflowproof}, Lamm--Simon~\cite{lamm2023ricci}, Chu--Lee~\cite{ChuLee}, Chu--Lee--Zhu~\cite{chuleezhu}, Lee--Liu~\cite{leeliu}, Lott \cite{LOTTSc}. 
In the distributional setting, we obtain the following result.

\begin{thm}
  \label{main2}
  For any $n\in \N$, $\si \in (0,\frac{1}{4})$ there exists $\ep_1(\si,n) > 0$ such that the following holds.
  Let $(M^n,h)$ be a smooth, connected, closed $n$-dimensional Riemannian manifold and let $g_0 \in W^{1,2+2\si}(M)$ be such that
  \begin{equation*}
    |g_0-h| \leq \ep_1(\si,n).
  \end{equation*}
  Assume further, for some $b \in \R$, that the scalar curvature of $g_0$ is not less than $b$ in the distributional sense, which we write as  $\Sc_{g_0} +b \geq 0$ (see Definition~\ref{LeeLeFloch}).

  Then the solution $g(t)_{t\in (0,T(n,K_0)]}$ of the $h$-Ricci--DeTurck flow constructed in \cite{simon2002deformation} (see Theorem~\ref{C0RicciDeTurckthm} of this paper) satisfies $g(t) \to g_0$ in $W^{1,2+\si}(M)$ as $t \downto 0$ and $\Sc_{g(t)} +b \geq 0$ (respectively $\Sc_{\ell(t)} + b \geq 0$ for a so-called \emph{related Ricci flow solution} $\ell(t)$, see Section~\ref{RDeTurckRicci}) for all $t \in (0,T)$ in the smooth sense.
\end{thm}

We note that, if $n \geq 8$, one cannot improve this result to the case that 
\begin{align*}
  \frac{1}{h} \leq g_0 \leq c h
\end{align*}
for some $c >1$, as we now explain.
The examples of Cecchini--Frenck--Zeidler in Theorem~A of \cite{CecFreZei} show that for any $n \geq 8$ there exists $g_0$ with $g_0 \in W^{1,3}(M,h)$, $\frac{1}{h}\leq g_0 \leq ch$ (we call such a metric an \emph{$L^{\infty}$-metric}) whose distributional scalar curvature is not less than $0$, but there are no smooth metrics $g$ on $M$ with $\Sc_g \geq 0$. Whether there is a Ricci--DeTurck flow starting at such a metric $g_0$ is to the best of our knowledge unknown. However, if such a solution did exist, then the examples in Theorem~A of \cite{CecFreZei} show that it cannot have scalar curvature bigger than or equal to $0$ for $t>0$.

As far as we know, the situation in dimension $n<8$ is still unclear. If  $g_0 \in W^{1,2+2\si}(M^n,h)$, $\frac{1}{h} \leq g_0 \leq ch$ for some $\si \in (0,\frac{1}{4})$, $c \in [1,\infty)$, and $g_0$ has distributional scalar curvature not less than $0$, then it is unknown for $n < 8$, (a) whether there exists a metric $g$ with $\Sc_g \geq 0$, or (b) whether there exists a Ricci--DeTurck flow starting from $g_0$, or (c) whether there exists a Ricci--DeTurck flow $g(t)_{t \in (0,T]}$ starting from $g_0$ with $\Sc_{g(t)} \geq 0$ in the smooth sense for $t>0$.

To the best of our knowledge, the following case is also  open:
Assume $g_0 \in W^{1,2}(M,h)$, where $(M,h)$ is a smooth closed Riemannian manifold, and $ (1-\ep_0(n)) h \leq g_0 \leq (1+ \ep_0(n))h$, or even $g_0$ continuous, and $g_0$ has scalar curvature larger than or equal to $b$ in the distributional sense (also the case $b\geq 0$ is interesting). We do not know if there exists a smooth metric $g$ on $M$ with $\Sc_g\geq 0$. Whether the smooth solution $(M,g(t))_{t \in (0,T)}$ of the $h$-Ricci--DeTurck flow with initial value $g_0$ coming from Theorem~\ref{main1} above satisfies $\Sc_{g(t)} \geq 0$ is also unknown to us.

\subsection*{Acknowledgements}
We thank Gianmichele Di Matteo for useful initial discussions on the $\sigma=0$ case of Theorem~\ref{W1siRicciDeTurck}.
Both authors were supported by the Special Priority Program SPP 2026 \enquote{Geometry at Infinity} of the German Research Foundation (DFG).

\section{Ricci--DeTurck flow and Ricci flow}
\label{RDeTurckRicci}
 
Writing the Ricci flow equation for $(M^n,\ell(t))_{t \in [0,S)}$ in local coordinates, one obtains \cite{hamilton1982threemanifolds}
\begin{align*}
  \label{eq:RF}
  \tag{RF}
  \partt \ell_{jk}(x,t)  &= - 2 \Rc(\ell)_{jk}(x,t) \\
  &= \ell^{pi}(\partial^2_{pi}  \ell_{jk} -\partial^2_{ik}\ell_{pj} -\partial^2_{pj}\ell_{ik} + \partial^2_{jk}\ell_{pi}) + \dots,
\end{align*}
where the dots refer to lower order terms.
Hamilton \cite{hamilton1982threemanifolds} showed that this system is not \emph{strictly} parabolic. That is, the symbol of this equation has some zero eigenvalues resulting from the fact that the equation is invariant under diffeomorphisms: Hamilton explained in \cite{hamilton1982threemanifolds} that if $\ell(t)$ is a solution of \eqref{eq:RF} and $\phi: N \to M$ a diffeomorphism, then $\phi^{*} \ell(t)$ is also a solution of the Ricci flow. Hence well established methods from partial differential equations for \emph{strictly} parabolic systems used to show existence, uniqueness, or smoothing properties are not immediately available.

Initially, Hamilton used the Nash--Moser inverse function theorem in \cite{hamilton1982threemanifolds} to prove existence of a solution. 
Later, DeTurck~\cite{deturck1983deforming} introduced the \emph{Ricci--DeTurck flow}, which is a strictly parabolic system. It can be used to give an alternative proof of the existence of a solution to the Ricci flow using more standard methods from parabolic equations.

We say that a smooth family $(M,g(t))_{t \in I}$ of Riemannian manifolds, for $I \subset \R$ an interval, is a \emph{solution of the $h$-Ricci--DeTurck flow}, or just \emph{Ricci--DeTurck flow} (if it is clear what the \emph{background metric} $h$ is), if $(M,h)$ is a smooth Riemannian manifold and 
\begin{align*}
  \label{Meq}
  \tag{RDTF}
  \partt {g}_{ij} = &~ {g}^{ab} \, \gradh_a \gradh_b {g}_{ij}
  - {g}^{kl} {g}_{ip} h^{pq} \Rm_{jkql}(h)
  - {g}^{kl} {g}_{jp} h^{pq} \Rm_{ikql}(h) \\
  &~ +\tfrac{1}{2} {g}^{ab} {g}^{pq} \left(\gradh_i {g}_{pa} \gradh_j {g}_{qb}
    + 2 \gradh_a {g}_{jp} \gradh_q {g}_{ib} - 2 \gradh_a {g}_{jp}
    \gradh_b {g}_{iq} \right. \\
  &~\left. \qqqqquad - 2 \gradh_j {g}_{pa} \gradh_b {g}_{iq}
    - 2 \gradh_i {g}_{pa} \gradh_b {g}_{jq} \right),
\end{align*}
in the smooth sense on $M\times I$ (see \cite{deturck1983deforming,shi1989deforming,hamilton1993formation,simon2002deformation}). 
This equation is now parabolic. DeTurck~\cite{deturck1983deforming} and Shi~\cite{shi1989deforming} considered the case that $h=g_0$ and $g_0$ is smooth. Hamilton~\cite{hamilton1993formation} explained that one can consider $g_0$ and $h$ smooth, but $h \neq g_0$. In \cite{simon2002deformation}, Simon considered the setting that $h$ is smooth, but $h$ is not necessarily equal to $g_0$, and $g_0$ is continuous or \emph{almost continuous} (see Definition~\ref{def-fair-controlled} below).
 
In many settings, for example when $M$ is closed, there is \emph{always} a Ricci flow solution $(M,\ell(t))_{t \in I}$ related to $(M,g(t))_{t \in I}$.
To obtain such a solution, one proceeds as follows: For a given $S \in I$ let $\Phi(t): M \to M$, $t \in I$ be the smooth family of diffeomorphisms with $\Phi(S) = \mathrm{Id}$ satisfying
\begin{equation}
  \label{ODEDe}
  \begin{split}
    \partt {\Phi}^{\al}(x,t) &= V^{\al}(\Phi(x,t),t), \quad \text{ for all } (x,t)\in M \times I, \\
      \Phi(x,S) &= x,
  \end{split}
\end{equation}
where
\begin{equation}
  {V}^{\al}(y,t) := -{g}^{\be \ga} \left( {\Gamma(g)}^{\al}_{\be \ga} - {\Gamma(h)}^{\al}_{\be \ga} \right) (y,t). \label{defnV}
\end{equation}
These solutions always exist when $M$ is closed due to the theory of ordinary differential equations (see for example
 \cite[Theorem~9.4.8]{Lee}).
The family $\ell(t) := (\Phi(t))^*g(t)_{t \in I}$ is then a smooth solution of the Ricci flow. We call any Ricci flow solution $\ell(t)$ a \emph{related Ricci flow solution to $g(t)$} if $\ell(t) = (\Phi(t))^*g(t)_{t \in I}$, where $\Phi$ solves \eqref{ODEDe} for some $S \in I$.

\section{Ricci--DeTurck flow of almost continuous \texorpdfstring{$L^2_{\loc}$}{L\textasciicircum{}2\textunderscore{}loc}-metrics}

For simplicity we sometimes write \eqref{Meq} as 
\begin{align}
  \label{RDT}
  \partt g_{ij} 
    = &~ g^{ab} \big( \gradh_a \gradh_b g_{ij} \big) + \big( h^{-1} * g^{-1} * g * \Rm(h) \big)_{ij} \\
    \nonumber &{} + \big( g^{-1} * g^{-1} * \gradh g * \gradh g \big)_{ij}.
\end{align}
 
In the paper \cite{simon2002deformation}, Simon considered the Ricci--DeTurck flow of given metrics $g_0$ which are continuous or \emph{almost continuous} in the following sense.

\begin{defi}
  \label{def-fair-controlled}
  Let $M$ be a smooth manifold without boundary and let $g_0$ be a Riemannian metric on $M$.
  \begin{enumerate}[label=(\roman*)]
    \item We say that $g_0$ is \emph{$\ep$-almost continuous} if there exists a smooth metric $h$ so that $(M,h)$ is complete with $\sup_M |\Riem(h)| < \infty$ and $(1-\ep) h \leq g_0 \leq (1+\ep) h$.
    \item We say that $g_0$ is \emph{almost continuous} if it is $\ep_0(n)$-almost continuous where $\ep_0(n)>0$ is the (small) constant appearing in Theorem~\ref{C0RicciDeTurckthm} (below). 
  \end{enumerate}
 \end{defi}

\begin{rem}
  From the estimates of Shi~\cite[Theorem~1.1]{shi1989deforming}, we see that if we are given $g_0$ and $h$ as in case (i) of Definition~\ref{def-fair-controlled}, with $\big( 1-\frac{\ep_0(n)}{2} \big) h \leq g_0 \leq \big( 1+\frac{\ep_0(n)}{2} \big) h$, then we may slightly deform $h$ to $\ti h$ so that $\sup_M |^{\ti h} \nabla^j \Riem(\ti h)| < \infty$, $j \in \N_0$, and $(1-\ep_0(n)) \ti h \leq g_0 \leq (1+\ep_0(n)) \ti h$ holds.
\end{rem}

\begin{rem}
  As we always consider $(M,h)$ with $\sup_M |\Rm(h)| = K_0 < \infty$, we see using the Gromov volume monotonicity property that balls $B_h(x_0,r)$ with $r \in (0,1)$ always satisfy $\vol(B_h(x_0,r))\leq c(n,K_0)$. We use this fact freely throughout the paper, without further comment.
\end{rem}

\begin{thm}[{Simon~\cite[Theorem~5.2]{simon2002deformation}}]
  \label{C0RicciDeTurckthm}
  For any $n \in \N$ there exists $\ep_0(n)>0$ such that the following holds.
  Let $(M^n,h)$ be a smooth, connected, complete $n$-dimensional Riemannian manifold such that $K_j := \sup_M |\gradh^j \Rm(h)| < \infty$ for all $j \in \N_0$.
  Assume further that $g_0$ is a Riemannian metric on $M$, for which $(1-\ep) h \leq g_0 \leq (1+\ep)h$ holds, where $0 < \ep \leq \ep_0(n)$.
  
  Then there exists $T=T(n,\ep,K_0) > 0$, monotone increasing in the second variable and decreasing in the third variable, and a smooth complete solution $(M,g(t))_{t \in (0,T]}$ of \eqref{Meq} such that $\sup_{K} |g(t)-g_0|_{h} \to 0$ as $t \downto 0$ for all sets $K \cc U$, for all sets $U$ where $g_0|_U$ is continuous, and furthermore
  \begin{align*}
   &(1 - 2 \ep) h \leq g(t) \leq (1 + 2 \ep) h, \\
    &|\gradh^j g(t)|^2 \leq \frac{c_j(n,h)}{t^j}
  \end{align*}
  for all $t \in (0,T)$ and $j \in \N_0$, on all of $M$. If $g_0$ is smooth, then (by Simon~\cite[Lemma~5.1]{simon2002deformation}) the solution may be smoothly extended to $(M,g(t))_{t \in [0,T]}$.
  Here $c_j(n,h) = a(j,n,\sup_M |\Riem(h)|,\ldots$, $\sup_M |\gradh^j \Riem(h)|)$ and $a(j,n,k_0,k_1, \ldots,k_j)$ is monotone increasing in each of the variables $k_0,k_1,\ldots,k_j$ for $j \in \N_0$, $n \in \N$.
\end{thm}

\begin{rem}
If the initial metric $g_0$ is in $C^1_{\loc}(M)$ then the covariant derivative of $g_0$ remains uniformly (in time) bounded on compact subsets of $M$. The same is true for higher order derivatives: see the proof of Lemma~2.1 in Simon \cite{simonLipschitz}, where the $C^1_{\loc}$ case is explained explicitly. The higher order cases follow as in Shi \cite[Lemma~4.2]{shi1989deforming}.
\end{rem}
 
In this paper we are interested in metrics which are $\ep$-almost continuous for $\ep \leq \ep_0(n)$ and in $L^{2}_{\loc}$.
Using the estimates of \cite{simon2002deformation} and the evolution equation for the metric, we obtain the following.

\begin{thm}
  \label{L2lemmas}
  For any $n \in \N$ there exists $\ep_0(n) > 0$ such that the following holds.
  Let $(M^n,h)$ be a smooth, connected $n$-dimensional Riemannian manifold (not necessarily complete) such that $K_j := \sup_M |\gradh^j \Rm(h)| < \infty$ for all $j \in \N_0$, $0 < r < \ell < \infty$ and $B_{h}(x_0,4\ell) \cc M$.

  \begin{enumerate}[label=(\roman*)]
    \item Assume that $g(t)_{t \in (0,T)}$, $T \leq 1$ is a smooth solution of \eqref{Meq} in $B_h(x_0,4r)$ and
      \begin{equation*}
        |g(t)-h| \leq \ep_0(n) \quad \text{ in } B_h(x_0,4r).
      \end{equation*}
      Then
      \begin{equation}
        \label{initgradeest}
        |\gradh^j g(t)|^2 \leq \frac{c_j(n,h,r)}{t^j} \quad \text{ in } B_h(x_0,r)
      \end{equation}
      for all $t \in (0,T)$.
    \item Let $g(i)(t)_{t \in [0,T]}$, $T \leq 1$ be smooth solutions of \eqref{Meq} in $B_h(x_0,4r)$ for all $i \in \N$ with 
      \begin{equation*}
        |g(i)(t)-h| \leq \ep_0(n) \quad \text{ in } B_h(x_0,4r) \text{ for all  } t \in [0,T]
      \end{equation*}
      and assume that there exists $g_0 \in L^2(B_h(x_0,4r))$ such that $g(i)(0) \to g_0$ as $i\to \infty$ in $L^2(B_h(x_0,4r))$. Then, after taking a subsequence,  $g(i) \to g$ as $i\to \infty$ smoothly on compact subsets of $B_h(x_0,4r) \times (0,T)$, $g$ solves \eqref{Meq} for all $t \in (0,T)$ and $g$ satisfies
      \begin{equation}
        \label{firstweakgradientconv}
        \int_{B_h(x_0,r)} |g(t)-g_0|^2 dh + \int_{0}^t \int_{B_h(x_0,r)} |\gradh g (s)|^2 dh \, ds \to 0
      \end{equation}
      as $t \downto 0$.
    \item Assume that $g(t)_{t \in [0,T)}$, $T \leq 1$ is a smooth solution of \eqref{Meq} in $B_h(x_0,4\ell)$ with 
      \begin{equation*}
        |g(t)-h| \leq \ep_0(n)
      \end{equation*}
      where $\ep_0(n)$ is sufficiently small.  
      Then we have
      \begin{align}
        &\int_{B_h(x_0,r)} |g(t)-g(0)|^2 dh + \int_0^t \int_{B_h(x_0,r)} |\gradh g(s)|^2 dh \, ds \\
        \nonumber &\quad \leq t \cdot c(n,r,\ell,K_0) \left( 1 + \int_{B_h(x_0,\ell)} |\gradh g(0)|^2 \right)
      \end{align}
      for all $t \in [0,T)$.
    \item Let $M^n$, $g_0$, $h$ be  as in Theorem~\ref{C0RicciDeTurckthm}, and assume further that  $g_0 \in L^2_{\loc}(M)$. Then the solution $g(t)_{t\in (0,T(n,K_0)]}$ of \eqref{Meq} constructed in Theorem~\ref{C0RicciDeTurckthm} satisfies
      \begin{equation}
        \label{firstweakgradientest}
        \int_{B_h(x_0,r)} |g(t)-g_0|^2 dh + \int_{0}^t \int_{B_h(x_0,r)} |\gradh g (s)|^2 dh \, ds \to 0
      \end{equation}
      as $t \downto 0$ for all $x_0 \in M$ and $r > 0$.
    \item Let $M^n$, $g_0$, $h$ be  as in Theorem~\ref{C0RicciDeTurckthm}, and assume further that     $g_0 \in W^{1,2}_{\loc}(M)$.  Then the solution $g(t)_{t\in (0,T(n,K_0)]}$  of \eqref{Meq} constructed in Theorem~\ref{C0RicciDeTurckthm} satisfies
    \begin{align}
      &\int_{B_h(x_0,r)} |g(t)-g_0|^2 dh + \int_0^t \int_{B_h(x_0,r)} |\gradh g(s)|^2 dh \, ds \\
      \nonumber &\quad \leq t \cdot c(n,r,\ell,K_0) \left( 1 + \int_{B_h(x_0,\ell)} |\gradh g_0|^2 \right) 
    \end{align}
    for all $x_0 \in M$, all $0<r<\ell$, and all $t \in (0,T]$.
  \end{enumerate}
\end{thm}

\begin{proof}
The proof of (i) is identical to the one given in the proof of Lemmata 4.1,~4.2 in \cite{simon2002deformation}: the proofs there are local, and require only $|g(t)-h| \leq \ep_0(n)$ and $\sup_M |\gradh^j \Riem(h)| = K_j < \infty$.

For the proof of (ii), let $\eta: B_h(x_0,4r) \to [0,\infty)$, $0 \leq \eta \leq 1$ be a smooth cut-off function with $|\gradh \eta|^2 \leq c(n,r,\ell) \eta$ and $\eta \equiv 1$ in $B_h(x_0,r)$, $\eta \equiv 0$ in $(B_h(x_0,\ell))^c$. Fix $j\in \N$. Then we compute
\begin{align*}
  &\frac{1}{2}\partt \int_M \eta |g(i)(t)-g(j)(0)|^2 dh \\
  &\quad = \int_M \eta h \left( \partt g(i)(t), g(i)(t)-g(j)(0) \right) dh \\
  &\quad = \int_M \eta h^{ks} h^{lv}\Big(  g(i)^{rm}(t) \gradh_{rm}^2 g(i)_{kl}(t) \\
  &\qqqqquad + (h^{-1} * g(i)(t)^{-1} * g(i)(t) * \Rm(h))_{kl} \\
  &\qqqqquad + (g(i)(t)^{-1} * g(i)(t)^{-1} * \gradh g(i)(t) * \gradh g(i)(t))_{kl} \Big) \\
  &\qqqquad \cdot (g(i)(t)-g(j)(0))_{sv} dh \\
  &\quad \leq \int_M \Big( -\eta \frac{1}{2} |\gradh g(i)(t)|^2 + \eta c(n) |\gradh g(i)(t)| |\gradh g(j)(0)| \\
  &\qqqqquad + \eta c(n,K_0) + |\gradh \eta| |\gradh g(i)(t)| \Big) dh \\
  &\quad \leq \int_{B_{h}(x_0,\ell)} \Big( c(n,r,\ell,K_0) - \eta \frac{1}{8} |\gradh g(i)(t)|^2 + \eta c(n) |\gradh g(j)(0)|^2 \Big) dh,
\end{align*}
where we used that $|g(i)(t)-g(j)(0)| \leq |g(i)(t)-h| + |h-g(j)(0)| \leq c(n) \ep_0(n) \leq \frac{1}{100}$ and $|\gradh \eta|^2 \leq c(n,r,\ell) \eta$.

Hence
\begin{align}
  \label{litleestin}
  &\int_{B_h(x_0,r)}  |g(i)(t)-g(j)(0)|^2 dh + \int_{0}^t \int_{B_h(x_0,r)}   |\gradh g(i)(s)|^2 dh \, ds \\
  \nonumber &\quad \leq  \int_{B_h(x_0,\ell)}   |g(i)(0)-g(j)(0)|^2 dh  + c_j(n,r,\ell,K_0) t
\end{align}
for all $t \in [0,T ]$, where $c_j(n,r,\ell,K_0) := c(n,r,\ell,K_0) \int_{B_{h}(x_0,\ell)} (|\gradh g(j)(0)|^2+1) dh$.
Taking a limit $i\to \infty$, that is, after taking a  subsequence of the solutions
$(B_{h}(x_0,4r),$ $g(i)(t))_{t \in (0,T]}$, choosing $\ell=2r$, and using the estimate \eqref{litleestin} and the estimates from part (i), we obtain a smooth solution $(B_{h}(x_0,4r),g(t))_{t \in (0,T]}$, satisfying the estimates \eqref{initgradeest}, such that
\begin{align*}
  & \int_{B_h(x_0,r)} |g(t)-g(j)(0)|^2 dh + \int_{\si}^t \int_{B_h(x_0,r)} |\gradh g(s)|^2 dh \, ds \\
  & \qquad \leq c_j(n,r,2r,K_0) t
  +  \int_{B_h(x_0,\ell)} |g_0-g(j)(0)|^2 dh
\end{align*}
for all $0<\si<t<T$, and hence
\begin{align*}
  & \int_{B_h(x_0,r)} |g(t)-g(j)(0)|^2 dh + \int_{0}^t \int_{B_h(x_0,r)} |\gradh g(s)|^2 dh \, ds \\
  & \qquad \leq c_j(n,r,2r,K_0) t
  +  \int_{B_h(x_0,\ell)} |g_0-g(j)(0)|^2 dh.
\end{align*}

Using this, and the triangle inequality, we see that
\begin{align*} 
  &\int_{B_h(x_0,r)} |g(t)-g_0|^2 dh + \int_{0}^t \int_{B_h(x_0,r)} |\gradh g(s)|^2 dh \, ds \\
  &\quad \leq \int_{B_h(x_0,r)} |g(t)-g(j)(0)|^2 dh + \int_{B_h(x_0,r)} |g_0-g(j)(0)|^2 dh \\
  &\qquad  + \int_{0}^t \int_{B_h(x_0,r)} |\gradh g(s)|^2 dh \, ds \\
  &\quad \leq c_j(n, r,2r,K_0) t + v_j,
\end{align*}
where $v_j = \frac{1}{j} + 2\int_{B_h(x_0,\ell)} |g_0-g(j)(0)|^2 dh \to 0$ as $j \to \infty$, and hence  
\begin{equation*}
  \int_{B_h(x_0,r)} |g(t)-g_0|^2 dh + \int_{0}^t \int_{B_h(x_0,r)} |\gradh g(s)|^2 dh \, ds \leq 2 v_j
\end{equation*}
for all $t \leq T_j = \frac{v_j}{c_j(n,r,2r,K_0)}$. Notice that $T_j>0$ (since $v_j>0$), so   the  estimate  holds for $t\in (0,T_j]$ and 
$(0,T_j]$ is non-empty. 
That is,
\begin{equation*} 
  \int_{B_h(x_0,r)} |g(t)-g_0|^2 dh + \int_{0}^t \int_{B_h(x_0,r)} |\gradh g(s)|^2 dh \, ds \to 0
\end{equation*}
as $t \downto 0$. This proves (ii).

The proof of (iii) is achieved by slightly modifying the proof of (ii):
the proof of (ii) shows that 
\begin{align*}
  &\frac{1}{2} \partt \int_M \eta |g(t)-g(0)|^2 dh  \\
  &\quad \leq \int_{B_{h}(x_0,2r)} \left( c(n,r,\ell,K_0) - \eta \frac{1}{8} |\gradh g(t)|^2 + \eta c(n) |\gradh g(0)|^2 \right) dh \\
  &\quad \leq c(n,r,\ell,K_0) \left( 1 + 2 \int_{B_h(x_0,\ell)} |\gradh g(0)|^2 dh \right) - \int_M \eta \frac{1}{8} |\gradh g(t)|^2 dh,
 \end{align*}
which, after integrating in time, implies the result.  
   
The statements (iv) and (v) can be proved by mollifying $g_0$ and applying the previous results, as we now explain.  
Mollifying $g_0$, we obtain $g(i)(0)$ with $|g(i)(0)-h|\leq c(n) \ep_0(n)$ (without loss of generality $\leq \ep_0(n)$) and $g(i)(0) \to g_0$ locally in $L^2$ in the case of (iv), respectively $g(i)(0) \to g_0$ locally in $W^{1,2}$ in the case of (v), as $i \to \infty$. 

Let $g(i)(t)$ be the solutions constructed in \cite{simon2002deformation} (see Theorem~\ref{C0RicciDeTurckthm}) and let $g(t)_{t \in (0,T)}$ be the solution obtained by taking a limit as  $i \to \infty$ of a subsequence of the sequence of solutions $(g(i)(t)_{t \in (0,T)})_{i \in \N}$. This solution agrees with the solution obtained in \cite{simon2002deformation}, as the same construction was used there. 
In the case that $g_0 \in L^2_{\loc}(M)$, the estimate in (ii) (applied to $g$) implies (iv). 
In the case that $g_0 \in W^{1,2}_{\loc}(M)$, we see that (v) holds by applying the estimates in (iii) to each $g(i)$ and then taking a limit $i \to \infty$. 
\end{proof}

\section{Ricci--DeTurck flow of almost continuous \texorpdfstring{$W^{1,2+2\si}_{\loc}$}{W\textasciicircum{}1,2+2σ\textunderscore{}loc}-metrics}

We show in this section that solutions to Ricci--DeTurck flow of almost continuous metrics which are in $W^{1,2+2\si}_{\loc}$, for $\si \in [0,\frac{1}{4})$, remain bounded uniformly in time in $W^{1,2+2\si}_{\loc}$ and approach their initial values in the $W^{1,2+\si}_{\loc}$-sense. We first prove estimates.

\begin{thm}
  \label{W1siRicciDeTurck}
  For any $n \in \N$ there exists $\ep_0(n) > 0$ such that the following holds.
  Let $(M^n,h)$ be a smooth, connected $n$-dimensional Riemannian manifold (not necessarily complete) with $K_j := \sup_M |\gradh^j \Rm(h)| < \infty$ for all $j \in \N_0$, and $ B_{h}(x_0,4R) \cc M$. 
  Suppose that $g(t)_{t \in [0,T]}$, $T\leq 1$  is a smooth solution of the Ricci--DeTurck flow on $B_{h}(x_0,R)$ which satisfies 
  \begin{equation*}
    |g(x,t)-h(x)| \leq \ep_0(n)
  \end{equation*}
  for all $t \in [0,T]$, $x \in B_{h}(x_0,R)$.
      
  Then it holds that
  \begin{align}
    \label{eq:SobolevEstimate2}
    &\int_{B_h(y_0,r)} |\gradh g(t)|^{2+2\si} dh \\
    \nonumber &\quad + \int_0^t \int_{B_h(y_0,r)} \left( |\gradh g(s)|^{4+2\si} + |\gradh g(s)|^{2\si} |\gradh^2 g(s)|^2 \right) dh \, ds \\
    \nonumber &\quad \leq c(n)\int_{B_h(y_0,\ell)} |\gradh g(0)|^{2+2\si} dh + t \cdot c(n,r,\ell,K_0,K_1) \left( 1+\int_{B_h(y_0,\ell)}|\gradh g(0)|^2 \right)    
  \end{align}
  for any  $y_0 \in B_{h}(x_0,R)$, $\si \in [0,\frac{1}{4}]$, $0<r<\ell<\de(n)$ with $B_h(y_0,4\ell) \subseteq  B_{h}(x_0,R)$,  $t\leq T$.
\end{thm}

\begin{proof}
For reasons of exposition, and for later use,  we prove the case $\si =0$ first. 
We scale the solution $g$, if necessary, by a constant $C(n,K_0,K_1) \geq 1$, so that $K_0 + K_1 = \sup_{M} |\Rm(h)| + |\gradh \Rm(h)| \leq 1$ and prove the estimates in this setting. Scaling back by the constant $C(n,K_0,K_1)^{-1}$ then implies the desired estimates.

We will calculate the evolution equation of
\begin{equation*}
  \omega(\cdot,t):= |\gradh g|^2(\cdot,t) (1+ L|g(\cdot,t)-h(\cdot)|^2)
\end{equation*}
for some suitably chosen large $L=L(n)$ and suitably chosen small $\ep_0(n)$, with $|g(t)-h| \leq \ep_0(n)$. We use a number of times in the proof that, without loss of generality, $L \ep_0(n)$ is small: We choose $L = \frac{1}{\sqrt{\ep_0(n)}}$, which implies $L \ep_0(n) \leq \sqrt{\ep_0(n)}$.

A quantity similar to $\omega$, albeit in a slightly different setting, was considered in the $L^2$-lemma of \cite{deruelle2021regularity}. There, the quantity $|g_1(t) - g_2(t)|^2 \big(1 + L(|g_1(t)-h|^2 + |g_2(t)-h|^2) \big)$ was considered, where $g_1$, $g_2$ are solutions which start from the almost continuous data $g_1(0) = g_2(0)$.

We use the following estimate of  $|\grad \omega|$ freely in the following, without further mention:
\begin{align*}
  |\grad \omega| &= |\grad (|\gradh g|^2) (1 + L |g-h|^2) + |\gradh g|^2 \, \grad (1 + L|g-h|^2)| \\
  &\leq c(n) \big( |\grad^2 g| |\grad g| + |\grad g|^3 \big).
\end{align*}

To begin, taking a derivative in $t$ gives
\begin{align}
  \label{evn1}
  \nonumber \pd{}{t} \abs{\gradh g}^2 
  &= \pd{}{t} h^{mn} h^{ik} h^{j\ell} (\gradh_m g_{ij} \gradh_n g_{k\ell}) \\
  \nonumber &= 2 h^{mn} h^{ik} h^{j\ell} \, \gradh_m (g^{ab} \, \gradh_{a} \gradh_{b} g_{ij}) \gradh_n g_{k\ell} \\
  \nonumber &\quad + 2 h^{mn} h^{ik} h^{j\ell} \, \gradh_m \big( (h^{-1} * g^{-1} * g * \Rm(h))_{ij} \\
  \nonumber &\qquad + (g^{-1} * g^{-1} * \gradh g * \gradh g)_{ij} \big) \gradh_n g_{k \ell} \\
  \nonumber &= 2 h^{mn} h^{ik} h^{j\ell} g^{ab} (\gradh_m \gradh_{a} \gradh_{b} g_{ij}) \gradh_n g_{k\ell} \\
  \nonumber &\quad + 2 h^{-1} * h^{-1} * h^{-1} * g^{-1} * g^{-1} * g^{-1} * \gradh g * \gradh^2 g * \gradh g \\
  \nonumber &\quad + 2 h^{-1} * h^{-1} * h^{-1} * h^{-1} * \Big( g^{-1} * g^{-1} * \gradh g * g * \Rm(h) \\
  \nonumber &\qquad + g^{-1} * \gradh g * \Rm(h) + g * g^{-1} * \gradh \Rm(h) \Big) * \gradh g \\
  \nonumber &\quad + 2 h^{-1} * h^{-1} * h^{-1} * \Big(g^{-1} * g^{-1} * \gradh^2 g * \gradh g \\
  \nonumber &\qquad + g^{-1} * g^{-1} * \gradh g * \gradh g * \gradh g \Big) * \gradh g \\
  \nonumber &\leq 2 h^{mn} h^{ik} h^{j\ell} g^{ab} (\gradh_m \gradh_{a} \gradh_{b} g_{ij}) \gradh_n g_{k \ell} \\
  \nonumber &\quad + c(n) |\gradh^2 g| |\gradh g|^2 + c(n)|\gradh g|^2 +c(n)|\gradh g| + c(n) |\gradh g|^4 \\
  &\leq 2 h^{mn} h^{ik} h^{j\ell} g^{ab} (\gradh_m \gradh_{a} \gradh_{b} g_{ij}) \gradh_n g_{k \ell} \\
  \nonumber &\quad + c(n) |\gradh^2 g| |\gradh g|^2 + c(n) |\gradh g|^2 + c(n) |\gradh  g| + c(n) |\gradh g|^4.
\end{align}

Taking derivatives in space, we obtain 
\begin{align*}
  \grad_b \abs{\grad g}^2 &= \grad_b (h^{mn} h^{ik} h^{j\ell} \, \grad_m g_{ij} \grad_n g_{k\ell}) \\
  &= h^{mn} h^{ik} h^{j\ell} (\grad_b \grad_m g_{ij} \grad_n g_{k\ell} + \grad_m g_{ij} \grad_b \grad_n g_{k\ell}),
\end{align*}
and hence 
\begin{align*}
  \grad_a \grad_b \abs{\grad g}^2 &= h^{mn} h^{ik} h^{j\ell} \, \grad_a (\grad_b \grad_m g_{ij} \grad_n g_{k\ell} + \grad_m g_{ij} \grad_b \grad_n g_{k\ell}) \\
  &= h^{mn} h^{ik} h^{j\ell} (\grad_a \grad_b \grad_m g_{ij} \grad_n g_{k\ell} + \grad_b \grad_m g_{ij} \grad_a \grad_n g_{k\ell} \\
  &\qquad + \grad_a \grad_m g_{ij} \grad_b \grad_n g_{k\ell} + \grad_m g_{ij} \grad_a \grad_b \grad_n g_{k\ell}) \\
  &= 2 h^{mn} h^{ik} h^{j\ell} (\grad_a \grad_b \grad_m g_{ij} \grad_n g_{k\ell} + \grad_b \grad_m g_{ij} \grad_a \grad_n g_{k\ell}).
\end{align*}
By swapping the order of the third order covariant derivatives appearing in this equation, we introduce curvature terms:  
\begin{align*}
  g^{ab} \, \grad_a \grad_b \abs{\grad g}^2 &= 2 h^{mn} h^{ik} h^{j\ell} g^{ab} (\grad_a \grad_b \grad_m g_{ij} \grad_n g_{k\ell} \\
  &\quad + \grad_b \grad_m g_{ij} \grad_a \grad_n g_{k\ell}) \\
  &= 2 h^{mn} h^{ik} h^{j\ell} g^{ab} \Big((\grad_m \grad_a \grad_b g_{ij} \\
  &\qquad + \Rm\indices{_{am}^p_b}(h) \grad_p g_{ij} + \Rm\indices{_{am}^p_i}(h) \grad_b g_{pj} \\
  &\qquad + \Rm\indices{_{am}^p_j}(h) \grad_b g_{ip} + \grad_a \Rm\indices{_{bm}^p_i}(h) g_{pj} \\
  &\qquad + \Rm\indices{_{bm}^p_i}(h) \grad_a g_{pj} + \grad_a \Rm\indices{_{bm}^p_j}(h) g_{ip} \\
  &\qquad + \Rm\indices{_{bm}^p_j} \grad_a g_{ip}) \grad_n g_{k\ell} \\
  &\quad + \grad_b \grad_m g_{ij} \grad_a \grad_n g_{k\ell}\Big),
\end{align*}
which implies, using $\sup_{B_h(x_0,4R)}  |\Rm(h)|+ |\grad \Rm (h)| \leq 1$, that
\begin{align*}
  & 2 h^{mn} h^{ik} h^{j\ell} g^{ab} ( \grad_m \grad_a \grad_b g_{ij} \grad_n g_{k\ell} ) \\
  &\quad \leq g^{ab} \, \grad_a \grad_b \abs{\grad g}^2 + c(n)(|\grad g| + |\grad g|^2) - \frac{3}{2} |\grad^2 g|^2 \\
  &\qquad + c(n) |\grad^2 g| |\grad g|^2 + c(n) |\grad g|^2 + c(n)|\grad g| + c(n) |\grad g|^4 \\
  &\quad \leq g^{ab} \, \grad_a \grad_b \abs{\grad g}^2 - \frac{4}{3} |\grad^2 g|^2 + c(n)(1+|\grad g|^4),
\end{align*}
and hence, using \eqref{evn1}, we see that
\begin{equation}
  \label{partgradhsquared} 
  \pd{}{t} \abs{\grad g}^2 \leq g^{ab} \, \grad_a \grad_b \abs{\grad g}^2 - |\grad^2 g|^2 + c(n)(1+|\grad g|^4).
\end{equation}

We calculate (see \cite{schnurer2008stability,schnurer2011stability}, where a similar calculation was performed)
\begin{align}
  \label{partgminushsquared}
  \nonumber \pd{}{t} \abs{g-h}^2 = {}& 2 h^{ik} h^{j\ell}  (g_{k\ell} - h_{k\ell}) \pd{}{t} g_{ij} \\
  \nonumber = {}& 2 h^{ik} h^{j\ell} (g_{k\ell} - h_{k\ell}) \Big(g^{ab} \, \grad_a \grad_b g_{ij} \\
  \nonumber {}& - g^{mn} g_{ip} h^{pq} \Rm_{jmqn}(h) - g^{mn} g_{jp} h^{pq} \Rm_{imqn}(h) \\
  \nonumber {}& + \frac{1}{2} g^{ab} g^{pq} (\grad_i g_{pa} \grad_j g_{qb} + 2 \grad_a g_{jp} \grad_q g_{ib} \\
  \nonumber &\quad - 2 \grad_a g_{jp} \grad_b g_{iq} - 2 \grad_j g_{pa} \grad_b g_{iq} - 2 \grad_i g_{pa} \grad_b g_{jq})\Big) \\
  \leq {}& g^{ab} \, \grad_a \grad_b \abs{g-h}^2 - \frac{3}{2} |\grad g|^2 + c(n)|\ep_0(n)|.
\end{align}

Using the estimates \eqref{partgradhsquared} and \eqref{partgminushsquared}, we obtain
\begin{align}
  \label{intermediate}
  \nonumber \pd{}{t} \omega &= \pd{}{t}\Big( (1+L\abs{g-h}^2) \abs{\grad g}^2 \Big) \\
  \nonumber &= (1+L\abs{g-h}^2) \pd{}{t} (\abs{\grad g}^2) + \abs{\grad g}^2 \pd{}{t} (L \abs{g-h}^2) \\
  \nonumber &\leq (1+L\abs{g-h}^2) \Big( g^{ab} \, \grad_a \grad_b \abs{\grad g}^2 - \frac{6}{5} |\grad^2 g|^2 + c(n)(1+|\grad g|^4) \Big) \\
  \nonumber &\quad + |\grad g|^2 \Big( g^{ab} \, \grad_a \grad_b (1 + L\abs{g-h}^2) - \frac{3}{2} L |\grad g|^2 + c(n) L|\ep_0(n)| \Big) \\
  &\leq g^{ab} \, \grad_a \grad_b \omega - 2 g^{ab} \, \grad_a (1 + L\abs{g-h}^2) \grad_b (\abs{\grad g}^2) \\
  \nonumber &\quad - \frac{7}{6} |\grad^2 g|^2 - L |\grad g|^4 + c(n)
\end{align}      
where we used Young's inequality and the fact that $L|g-h|^2 \leq L \ep^2_0(n)c(n) \leq \ep_0(n) \ll 1$.
The term $-2 g^{ab} \, \grad_a (1 + L\abs{g-h}^2) \grad_b (\abs{\grad g}^2)$ appearing in this inequality may be estimated as follows:
\begin{align*}
  |g^{ab} \, \grad_a (1 + L\abs{g-h}^2) \grad_b (\abs{\grad g}^2)| &\leq c(n) L \ep_0(n) |\grad g|^2 |\grad^2 g| \\
  &\leq c(n) \sqrt{\ep_0(n)} |\grad g|^4 + c(n)\sqrt{\ep_0(n)} |\grad^2 g|^2.
\end{align*}
Inserting this into \eqref{intermediate}, we get
\begin{align}
  \label{finalomega}
  \pd{}{t} \omega \leq g^{ab} \, \grad_a \grad_b \omega - \frac{8}{7} |\grad^2 g|^2 - \frac{3}{4}L |\grad g|^4 + c(n).
\end{align}

Let $\eta$ be a cut-off function (in space) satisfying $0 \leq \eta \leq 1$, $\eta \equiv 1$ in $B_h(y_0,r)$, $\eta \equiv 0$ in $B_h(y_0,\frac{r+\ell}{2})^c$, and $\frac{|\grad \eta|^2}{\eta} + \frac{|\grad \eta|^4}{\eta^3} + \frac{|\grad^2 \eta|^2}{\eta} \leq c(n,r,\ell)$. 
Cut-off functions of this type always exist, and can be constructed using
Theorem~1 of \cite{tam2010exhaustion} (or \cite[Lemma~12.30]{chowetal}) and a scaling argument, replacing $\eta$ by $\eta^{1000}$ if necessary. 
Using the estimate \eqref{finalomega}, we obtain 
\begin{align}
  \label{finalomegacutoff}
  \nonumber \pd{}{t} (\eta \omega) &\leq g^{ab} \, \grad_a \grad_b (\eta \omega) +c(n) \eta - \frac{8}{7} \eta |\grad^2 g|^2 \\
  \nonumber &\quad - \frac{3}{4} L \eta |\grad g|^4 - \omega g^{ab} \, \grad_a \grad_b \eta - 2 g^{ab} \, \grad_a \eta \grad_b \omega \\
  \nonumber &\leq g^{ab} \, \grad_a \grad_b (\eta \omega) + c(n) \eta - \frac{8}{7} \eta |\grad^2 g|^2 \\
  \nonumber &\quad - \frac{3}{4} L \eta |\grad g|^4 + c(n) |\grad^2 \eta| \omega + c(n) |\grad \eta| \Big(|\grad^2 g| |\grad g| + |\grad g|^3\Big) \\
  \nonumber &\leq g^{ab} \, \grad_a \grad_b (\eta \omega) + c(n) \eta \\
  \nonumber &\quad + c(n) \left(\frac{|\grad \eta|^2}{\eta} + |\grad^2 \eta|\right) |\grad g|^2 - \frac{L}{2} \eta |\grad g|^4 - \frac{1}{2} \eta |\grad^2 g|^2 \\
  &\leq  g^{ab} \, \grad_a \grad_b (\eta \omega) - \frac{L}{4} \eta |\grad g|^4 \\
  \nonumber &\quad - \frac{1}{4} \eta |\grad^2 g|^2 + c(n) \left(\eta + \frac{|\grad \eta|^4}{\eta^3} + \frac{|\grad^2 \eta|^2}{\eta}\right).
\end{align}
Integrating in space and using integration by parts yields
\begin{align*}
  &\partt \int_M (\eta \omega) dh \\
  &\quad \leq \int_M \Big( g^{ab} \, \grad_a \grad_b (\eta \omega) - \frac{L}{4} \eta |\grad g|^4 \\
  &\qqqqquad - \frac{1}{4} \eta |\grad^2 g|^2 + c(n,r,\ell)\chi_{B_h(x_0,\frac{r+\ell}{2})} \Big) dh \\
  &\quad \leq \int_M \Big( c(n) |\grad g| \big(|\grad \eta| \omega + \eta |\grad \omega| \big) - \frac{L}{4} \eta |\grad g|^4 \\
  &\qqqqquad - \frac{1}{4} \eta  |\grad^2 g|^2 + c(n,r,\ell)\chi_{B_h(x_0,\frac{r+\ell}{2})} \Big) dh \\
  &\quad \leq \int_M \Big( c(n) |\grad g| \big(|\grad \eta| |\grad g|^2 + \eta (|\grad g|^3 + |\grad^2 g| |\grad g|) \big) - \frac{L}{4} \eta |\grad g|^4 \\
  &\qqqqquad - \frac{1}{4} \eta |\grad^2 g|^2 + c(n,r,\ell) \chi_{B_h(x_0,\frac{r+\ell}{2})} \Big) dh \\
  &\quad \leq \int_M \Big( c(n,r,\ell)\chi_{B_h(y_0,\frac{r+\ell}{2})} - \frac{L}{8} \eta |\grad g|^4 - \frac{1}{10} \eta |\grad^2 g|^2    \Big) dh.
\end{align*}
Integrating in time leads to the estimate \eqref{eq:SobolevEstimate2} in the case $\si = 0$.
   
To obtain the  estimate \eqref{eq:SobolevEstimate2} for general $\si \in [0,\frac{1}{4}]$, we modify this argument slightly.
For $\al \in (0,1)$, $\si \in (0, \frac{1}{4})$ our estimate \eqref{finalomega} implies that
\begin{align*}
  \nonumber \pd{}{t} (\al+\omega)^{1+\si} &\leq g^{ab} \, \grad_a \grad_b (\al+\omega)^{1+\si} + c(n)(\al+\omega)^{\si} \\
  \nonumber &\quad - (1+\si)\si(\al+\omega)^{\si-1} g^{ab} \, \grad_a \omega \grad_b \omega - \frac{8}{7}(1+\si) (\al+\omega)^{\si} |\grad^2 g|^2 \\
  \nonumber &\quad - (1+\si) (\al+\omega)^{\si} \frac{3}{4}L |\grad g|^4 \\
  &\leq g^{ab} \, \grad_a \grad_b (\al+\omega)^{1+\si} + c(n)(\al+\omega)^{\si} \\
  \nonumber &\quad - \frac{8}{7}(1+\si) (\al+\omega)^{\si} |\grad^2 g|^2 - (1+\si) (\al+\omega)^{\si}\frac{3}{4}L |\grad g|^4,
\end{align*}
and hence, arguing as in \eqref{finalomegacutoff}, we see that
\begin{align*}
  \nonumber \pd{}{t} (\eta (\al+\omega)^{1+\si}) &\leq g^{ab} \, \grad_a \grad_b (\eta (\al+\omega)^{1+\si}) + c(n) \eta (\al+\omega)^{\si} \\
  \nonumber &\quad + c(n) (1+\si) |\grad^2 \eta| (\al+\omega)^{1+\si} \\
  \nonumber &\quad + (1+\si) c(n) (\al+\omega)^{\si} |\grad \eta| |\grad \omega| \\
  \nonumber &\quad - (1+\si) \eta (\al+\omega)^{\si} |\grad^2 g|^2 - (1+\si) \eta (\al+\omega)^{\si} \frac{3}{4}L |\grad g|^4 \\
  \nonumber &\leq g^{ab} \, \grad_a \grad_b (\eta (\al +\omega)^{1+\si}) + c(n) \eta (\al+\omega)^{\si} \\
  \nonumber &\quad + c(n) |\grad^2 \eta| (\al+\omega)^{\si} (1 + |\grad g|^2) \\
  \nonumber &\quad + c(n) (\al+\omega)^{\si} |\grad \eta| \big(|\grad^2 g| |\grad g| + |\grad g|^3\big) \\
  \nonumber &\quad - (1+\si) \eta (\al+\omega)^{\si} |\grad^2 g|^2 - (1+\si) \eta (\al+\omega)^{\si} \frac{3}{4}L |\grad g|^4 \\
  \nonumber &\leq g^{ab} \, \grad_a \grad_b (\eta (\al + \omega)^{1+\si}) \\
  \nonumber &\quad + c(n) (\al+\omega)^{\si} \left(\frac{|\grad^2 \eta|^2}{\eta}+\chi_{B_h(y_0,\frac{r+\ell}{2})}\right)\\
  \nonumber &\quad + c(n) \frac{|\grad \eta|^2}{\eta} (\al+\omega)^{\si} |\grad g|^2 \\
  \nonumber &\quad - \frac{1}{2} (1+\si) \eta (\al+\omega)^{\si} |\grad^2 g|^2 - \frac{1}{2} (1+\si) \eta (\al+\omega)^{\si} L |\grad g|^4 \\
  &\leq g^{ab} \, \grad_a \grad_b (\eta (\al +\omega)^{1+\si}) \\
  \nonumber &\quad + c(n) (\al+\omega)^{\si} \left(\frac{|\grad^2 \eta|^2}{\eta} + \chi_{B_h(y_0,\frac{r+\ell}{2})} + \frac{|\grad \eta|^4}{\eta^3}\right) \\
  \nonumber &\quad - \frac{1}{4} (1+\si) \eta (\al+\omega)^{\si} |\grad^2 g|^2 - \frac{1}{4} (1+\si)\eta (\al+\omega)^{\si} L |\grad g|^4.             
\end{align*}

We integrate over space, and using integration by parts we obtain
\begin{align*}
  &\pd{}{t} \int_M \eta (\al+\omega)^{1+\si} dh \\
  &\quad \leq \int_M \Big( |\grad g| \big( |\grad \eta| (\al+\omega)^{1+\si} + c(n) \eta (1+\si) (\al+\omega)^{\si} (|\grad g|^3 + |\grad^2 g| |\grad g|) \big) \\
  &\qqqqquad - \frac{1}{4} (1+\si) \eta (\al+\omega)^{\si} |\grad^2 g|^2 - \frac{1}{4} (1+\si) \eta (\al+\omega)^{\si} L |\grad g|^4 \\
  &\qqqqquad + c(n) (\al+\omega)^{\si} \left( \frac{|\grad^2 \eta|^2}{\eta} + \chi_{B_h(y_0,\frac{r+\ell}{2})} + \frac{|\grad \eta|^4 }{\eta^3} \right) \Big) dh \\
  &\quad \leq \int_M \Big( c(n) \left( \frac{|\grad^2 \eta|^2}{\eta} +\chi_{B_h(y_0,\frac{r+\ell}{2})} + \frac{|\grad \eta|^4 }{\eta^3} \right) (\al+\omega)^{\si} \\
  &\qqqqquad - \frac{1}{8} (1+\si) \eta (\al+\omega)^{\si} |\grad^2 g|^2 - \frac{1}{8} (1+\si) \eta (\al+\omega)^{\si} L |\grad g|^4 \Big) dh.
\end{align*}
Integrating the above evolution inequality in time, using $(\al+\omega)^{\si} \leq  (1+|\grad g|)^{2\si} \leq (1+|\grad g|)^{2}$ and letting $\al \to 0$ yields       
\begin{align*}
  &\int_{B_h(y_0,r)} |\grad g|^{2+2\si} dh + \int_0^t \int_{B_h(y_0,r)} \left( |\grad g(s)|^{4+2\si} + |\grad g(s)|^{2\si} |\grad^2 g(s)|^2 \right) dh \, ds \\
  &\quad \leq \int_{B_h(y_0,\frac{\ell+r}{2})} |\grad g(0)|^{2+2\si} dh + c(n,r,\ell) \int_0^t \int_{B_h(y_0,\frac{\ell+r}{2})} (1+|\grad g|^{2}) dh.
\end{align*}
Using (v) from Theorem~\ref{L2lemmas}, we see that the right hand side may be estimated by $\int_{B_h(y_0,\ell)} |\grad g(0)|^{2+2\si} dh + t \cdot c(n,r,\ell,K_0)(1+\int_{B_h(y_0,\frac{\ell+r}{2})} |\grad g(0)|^{2}dh)  $.
Scaling back implies the   estimate \eqref{eq:SobolevEstimate2} for general $\si \in [0,\frac{1}{4}]$.
\end{proof}

Using these estimates, and arguments inspired by the ones given in \cite{lamm2023ricci}, we see that the solution from Theorem~\ref{C0RicciDeTurckthm} constructed in \cite{simon2002deformation} approaches the initial data in the $W^{1,2+\si}_{\loc}(M)$-sense if the initial data is $W^{2+2\si}_{\loc}(M)$ and $\si >0$, as we now explain.

\begin{thm}
  \label{W1siDeturckthm}
  For any $n\in \N$ there exists $\ep_0(n)>0$ such that the following holds. 
  Let $(M^n,h)$ be a smooth, connected, complete $n$-dimensional Riemannian manifold such that $K_j:= \sup_M |\gradh^j \Rm(h)| < \infty$ for all $j \in \N_0$. 
  Assume further that $g_0$ is a Riemannian metric on $M$ for which $(1-\e_0(n)) h \leq g_0 \leq (1+\ep_0(n)) h$ holds,  and that $g_0 \in W^{1,2+2\si}_{\loc}(B_{h}(x_0,2R))$ for some $\si \in [0,\frac{1}{4})$.

  Then the solution $g(t)_{t \in (0,T(n,K_0)]}$ of the Ricci--DeTurck flow from Theorem~\ref{C0RicciDeTurckthm} constructed in \cite{simon2002deformation} satisfies 
  \begin{align}
  \label{bettergradest}
    &\int_{B_h(y_0,r)} |\gradh g|^{2+2\si} dh + \int_0^t \int_{B_h(y_0,r)} \left( |\gradh g(s)|^{4+2\si} + |\gradh g(s)|^{2\si} |\gradh^2 g(s)|^2 \right) dh \, ds \\
    \nonumber &\quad \leq c(n)\int_{B_h(y_0,\ell)} |\gradh g_0|^{2+2\si} dh +  t \cdot c(n,r,\ell,K_0,K_1) \left( 1+\int_{B_h(y_0,\ell)}|\gradh g_0|^2 \right),
  \end{align}
  and
  \begin{equation}
    \label{Lpconvergence}
    \int_{B_{h}(x_0,R)} |g(t) - g_0|^{p}dh \to 0
  \end{equation}
  as $t \downto 0$
  for any $2\leq p < \infty$,
  for any $y_0 \in B_{h}(x_0,R)$, $0<r<\ell<1$ with $B_h(y_0,4\ell) \subseteq B_{h}(x_0,2R)$, $t \leq T$.

  If we further restrict to $\si\in (0, \frac{1}{4})$, then we also have
  \begin{align}
    \label{bettergradest2}  
    \int_{B_h(y_0,r)} |\gradh g|^{2} dh &\leq \int_{B_h(y_0,\ell)}  \abs{\grad g_0}^2 dh \\
    \nonumber &\quad + t^{v(\si)} c(n,r,\ell,K_0,K_1) \left( 1 + \int_{B_h(y_0,\ell)} \abs{\grad g_0}^{2+2\si} dh \right)
  \end{align}
  for any $y_0 \in B_{h}(x_0,R)$, $0<r<\ell<1$ with $B_h(y_0,4\ell) \subseteq B_{h}(x_0,R)$, $t \leq T$
  where $1>v(\si)>0$ is a small constant depending on $\si>0$, 
  and 
  \begin{align*}
    &\int_{B_{h}(x_0,R)} |\gradh g(t)-\gradh g_0|^{2+\si}dh \to 0 
  \end{align*}
  as $t \downto 0$.  
\end{thm}

\begin{proof}
The solution $g(t)$ constructed in \cite{simon2002deformation} is obtained as follows. We start by mollifying the initial value $g_0$ to obtain a sequence of initial values $g_i(0)$ with $|g_i(0)- h| \leq \ep(n)$, and $g_i(0) \to g_0$ in $W^{1,2+2\si}_{\loc}$. Then we use Theorem~\ref{C0RicciDeTurckthm} and Theorem~\ref{W1siRicciDeTurck} to obtain solutions $g_i(t)_{t \in [0,T(n,K_0)]}$ satisfying the estimates of Theorem~\ref{C0RicciDeTurckthm} and Theorem~\ref{W1siRicciDeTurck}. We proceed to take a limit to obtain a solution $g(t)_{t \in (0,T(n,K_0)]}$, and the estimates given in Theorem~\ref{C0RicciDeTurckthm} and Theorem~\ref{W1siRicciDeTurck} hold. 
In particular, the estimate \eqref{bettergradest} holds. 
Conclusion (v) of Theorem~\ref{L2lemmas} tells us that $\int_{B_{h}(x_0, R)} |g(t)-g_0|^{2} dh \to 0$ as $t \downto 0$.
Hence
\begin{align*}
  \int_{B_{h}(x_0, R)} |g(t)-g_0|^{p} dh \leq c(n,K_0)^{p-2} \int_{B_{h}(x_0, R)} |g(t)-g_0|^{2} dh \to 0
\end{align*}
as $t \downto 0$, so \eqref{Lpconvergence} is established.

We now proceed to show that the estimate \eqref{bettergradest2} holds. We scale, if necessary, so that $K_0 + K_1 \leq 1$. 
We consider the smooth solutions $g_i(t)$ with $g_i(0)$ being the mollified $g_0$ (as above), but for the sake of exposition we write $g(t)$ for $g_i(t)$.
From \eqref{partgradhsquared} we see 
\begin{equation}
  \label{part2gradhsquared}
  \pd{}{t} \abs{\grad g}^2 \leq g^{ab} \, \grad_a \grad_b \abs{\grad g}^2 - |\grad^2 g|^2 + c(n)(1+|\grad g|^4).
\end{equation}
  
Let $\eta$ be a cut-off function (in space) satisfying $0 \leq \eta \leq 1$, $\eta \equiv 1$ in $B_h(y_0,r)$, $\eta \equiv 0$ in $B_h(y_0,\frac{r+\ell}{2})^c$, and $\frac{|\grad \eta|^2}{\eta} + \frac{|\grad \eta|^4}{\eta^3} + \frac{|\grad^2 \eta|^2}{\eta} \leq c(n,r,\ell)$. 
Then, integrating \eqref{part2gradhsquared} over space and using integration by parts, we obtain 
\begin{align*}
  &\pd{}{t} \int_M \eta \abs{\grad g}^2 dh \\
  &\quad \leq \int_M \big( \eta g^{ab} \, \grad_a \grad_b \abs{\grad g}^2 - \eta|\grad^2 g|^2 + c(n)\eta(1+|\grad g|^4) \big) dh \\
  &\quad \leq \int_M \big( c(n)|\gradh \eta|  |\grad^2 g| |\grad g|  + c(n) \eta |\grad^2 g| |\grad g|^2  \\
  &\qqqquad - \eta|\grad^2 g|^2 + c(n)\eta(1+|\grad g|^4) \big) dh \\
  &\quad \leq \int_M \big( c(n)|\gradh \eta| |\grad^2 g| |\grad g| - \eta|\grad^2 g|^2 + c(n)\eta(1+|\grad g|^4) \big) dh \\
  &\quad \leq \int_M \Big( c(n)\frac{|\gradh \eta|^2}{\eta}|\grad g|^2 - \frac{3}{4} \eta|\grad^2 g|^2 + c(n)\eta(1+|\grad g|^4) \Big )dh \\
  &\quad \leq \int_M \Big( c(n)\frac{|\gradh \eta|^4}{\eta^3} - \frac{1}{2} \eta|\grad^2 g|^2 + c(n)\eta(1+|\grad g|^4) \Big) dh \\
  &\quad \leq c(n,r,\ell) + c(n)  \int_{B_h(y_0,\frac{r + \ell}{2})}  |\grad g|^4 dh.
\end{align*}

In the following, we choose constants $1<\hat r(\si), \hat v(\si) < \infty$ with
$\hat r(\si) = \frac{4+2\si}{4}$ and $\frac{1}{\hat r(\si)} + \frac{1}{\hat v(\si)} = 1$ and we define
$r(\si) = \frac{1}{\hat r(\si)}$ and $v(\si) = \frac{1}{\hat v(\si)}$ (note that $r(\si),v(\si)<1$).
Integrating in time, we get
  \begin{align*}
  &\int_{B_h(y_0,r)}  \abs{\grad g}^2 dh \\
  &\leq \int_{B_h(y_0,\frac{r+ \ell}{2})} \abs{\grad g(0)}^2 dh + c(n,r,\ell)t + c(n) \int_0^t \int_{B_h(y_0, \frac{r+ \ell}{2})} |\grad g|^4(x,s) dh(x) ds \\
  &\leq \int_{{B_h(y_0, \frac{r + \ell}{2})}} \abs{\grad g(0)}^2 dh + c(n,r,\ell) t \\
  &\quad + c(n) \Big(\int_0^t \int_{B_h(y_0, \frac{r + \ell}{2})} |\grad g|^{4+2\si}(x,s) dh(x)ds \Big)^{r(\si)} \Big(\int_0^t \int_{B_h(y_0,  \frac{r+ \ell}{2})} 1\, dh ds \Big)^{v(\si)}\\
  &\leq \int_{B_h(y_0,\ell)} \abs{\grad g(0)}^2 dh + c(n,r,\ell) t \\
  &\quad + c(n,r,\ell) \Bigg( \int_{B_h(y_0,\ell)} \abs{\grad g(0)}^{2+2\si} dh \\
  &\qqqquad + c(n,R,K_0,K_1) \Big( 1 + \int_{B_h(y_0,\ell)} |\grad g_0|^2 dh \Big) \Bigg)^{r(\si)} t^{v(\si)} \\
  &\leq \int_{B_h(y_0,\ell)} \abs{\grad g(0)}^2 dh + c(n,r,\ell,R,K_0,K_1) t^{v(\si)} \Big( 1 + \int_{B_h(y_0,\ell)}  \abs{\grad g(0)}^{2+2\si} dh \Big).
\end{align*}
That is, this inequality holds for the smooth solutions $g_i(t)$.
Letting $i\to \infty$, we see that the estimate \eqref{bettergradest2} holds.

Using a similar argument to the one used in the proof of Theorem~5.8 of \cite{lamm2023ricci}, we see, for $\si>0$, that the estimates \eqref{bettergradest} and \eqref{bettergradest2} can be used to obtain
\begin{equation*}
  \int_{B_{h}(x_0,\frac{6}{5}R)} |\gradh g(t) - \gradh g_0|^2 \to 0 \quad \text{ as } t \downto 0.
\end{equation*}
For the reader's convenience we reproduce the argument, suitably modifying the spaces so that they are appropriate to the setting being considered here.
We consider the Hilbert space $(H,\langle \cdot , \cdot \rangle_H)$ with $H = \{ \text{metrics } g \text{ defined in } B_{h}(x_0,\frac{5}{4}R) \text{ with } |g-h| \leq \ep(n) \text{ and } g \in W^{1,2}(B_{h}(x_0,\frac{5}{4}R)) \}$ and 
\begin{equation*}
  \langle g,k \rangle_H := \int_{B_{h}(x_0,\frac{5}{4}R)} h(g,k) + h(\gradh g , \gradh k) dh. 
\end{equation*}

Taking any sequence $(g(t_i))$ with $t_i \downto 0$, we see that $(g(t_i))$ is uniformly bounded (independent of $i$) in $(H,\langle \cdot , \cdot \rangle_{H})$, and hence there is a subsequence $(g(t_{i_j}))$ which weakly converges to a value $\ell \in H$ as $j \to \infty$ with respect to  $\langle \cdot , \cdot \rangle_{H}$.
Due to the fact that $\int_{B_{h}(x_0,\frac{7}{4} R)} |g(t)-g_0|^{2} dh \to 0$ as $t \downto 0$, we see that $\ell = g_0$.
This implies
\begin{equation}
  \label{Hilbert}
  \int_{B_{h}(x_0,\frac{5}{4}R)} h(\gradh g(t_{i_j}), \gradh g_0) dh \to \int_{B_{h}(x_0,\frac{5}{4}R)} |\gradh g_0|_h^2 dh \quad \text{ as } j \to \infty.
\end{equation}
For any $\ep>0$, we can find $\frac{5}{4}R \leq S \leq \frac{4}{3}R$ such that 
\begin{equation*}
  \int_{B_{h}(x_0, S)} |\gradh g_0|^2 dh \leq \int_{B_{h}(x_0,\frac{5}{4}R)} |\grad g_0|^2 dh + \ep.
\end{equation*}
But then, combining this with \eqref{Hilbert} and \eqref{bettergradest2}, we obtain 
\begin{align*}
  &\int_{B_{h}(x_0,\frac{5}{4}R)} |\gradh g(t_{i_j}) - \gradh g_0|_h^2 dh \\
  &\quad = \int_{B_{h}(x_0, \frac{5}{4}R)} \Big( |\gradh g(t_{i_j})|^2 + |\gradh g_0|^2 - 2 h(\gradh g(t_{i_j}), \gradh  g_0) \Big) dh \\
  &\quad \leq \int_{B_{h}(x_0, S)} |\gradh g_0|^2 dh +  t^{v(\si)}_{i_j} C(g_0,R,S) \\
  &\qquad + \int_{B_{h}(x_0, \frac{5}{4}R)} \Big(|\gradh g_0|^2 - 2h(\gradh g(t_{i_j}), \gradh g_0) \Big) dh \\
  &\quad \leq \ep + \int_{B_{h}(x_0,\frac{5}{4}R)} |\gradh g_0|^2 dh +   t^{v(\si)}_{i_j} C(g_0,R,S) \\
  &\qquad + \int_{B_{h}(x_0, \frac{5}{4}R)} \Big( |\gradh g_0|^2 - 2h(\gradh g(t_{i_j}), \gradh g_0) \Big) dh \\
  &\quad \to \ep \quad \text{ as } j \to \infty.
\end{align*}
That is, for any sequence $t_i \to 0$, we can find a subsequence $(t_{i_{j}})$ such that $g(t_{i_{j}}) \to g_0$ in $(H,\langle \cdot, \cdot \rangle_H)$ as $j \to \infty$ ($\ast$).
But this shows that $g(t) \to g_0$ in $(H,\langle \cdot, \cdot \rangle_H)$ as $t \downto 0$: If this were not the case, then we could find a number $\de >0$ and a sequence $t_i \downto 0$ such that $\langle g(t_i) - g_0 , g(t_i) - g_0 \rangle_H \geq \de$ as $i \to \infty$, which would contradict ($\ast$).

Therefore,
\begin{equation*}
  \int_{B_{h}(x_0,\frac{6}{5}R)} |g(t)-g_0|^p + |\gradh g(t) - \gradh g_0|^2 \to 0 \quad \text{ as } t \downto 0
\end{equation*}
for any $p \geq 2$.
For initial data $g_0 \in W^{1,2+2\si}(B_{h}(x_0,2R))$ we see that this implies
\begin{align*}
  &\int_{B_{h}(x_0,\frac{6}{5}R)} |\gradh g(t) - \gradh g_0|^{2+\si} dh \\
  &\quad = \int_{B_{h}(x_0,\frac{6}{5}R)} |\gradh g(t) - \gradh g_0| \, |\gradh g(t) - \gradh g_0|^{1+\si} dh \\
  &\quad \leq \left( \int_{B_{h}(x_0,\frac{6}{5}R)} |\gradh g(t) - \gradh g_0|^2 \right)^{\frac{1}{2}} \left( \int_{B_{h}(x_0,\frac{6}{5}R)} |\gradh g(t) - \gradh g_0|^{2+2\si} dh \right)^{\frac{1}{2}} \\
  &\quad \leq \left( \int_{B_{h}(x_0,\frac{6}{5}R)} |\gradh g(t) - \gradh g_0|^2 \right)^{\frac{1}{2}} \\
  &\qqquad \cdot \left( \int_{B_{h}(x_0,\frac{6}{5}R)} c \big(|\gradh g(t)|^{2+2\si} + |\gradh g_0|^{2+2\si} \big) dh \right)^{\frac{1}{2}} \\
  &\quad \to 0 \quad \text{ as } t \downto 0. \qedhere
\end{align*}
\end{proof}

\section{Preservation of weak scalar curvature bounds from below}
\label{scalarsection}

The method we use to show Theorem~\ref{main2} is the one used by Jiang--Sheng--Zhang~\cite{jiang2023weak} (and also, in a modified form, by Burkhardt-Guim~\cite{burkhardt-guim2024smoothing}). This involves testing the scalar curvature with solutions to the conjugate heat equation, and   proving estimates on the solutions.  See below for more details. 
In the paper of Jiang--Sheng--Zhang, the metrics $g_0$ considered are in $W^{1,p}(M)$ where $p > \frac{n}{2}$. In particular, there, one obtains that the solution to the Ricci--DeTurck flow satisfies $|\Rm(g(t))| \leq \frac{c}{t^r}$ (and $|\Rm(\ell(t)) |\leq \frac{c}{t^r}$ for the related Ricci flow solution $\ell(t)$) for some $0 < r(p,n) < 1$, which is integrable in time. This plays an important role in many estimates of that paper.
As this estimate is not available to us, we must prove other estimates in the class we are considering which still lead to the desired result.

We note here that the use of solutions of the conjugate heat equation in a Ricci flow background setting to obtain bounds on the scalar curvature is not new, and was used in papers appearing before \cite{jiang2023weak} and \cite{burkhardt-guim2024smoothing}: see for example \cite{bamlerricciflowproof}.

We recall now the method used in \cite{jiang2023weak} to show that bounds from below in the distributional sense will be preserved by the Ricci flow. Let $(M^n,\ell(t))_{t \in (0,T]}$ be a closed, smooth solution to the Ricci flow and $0<Y<T$.
Let $\varphi_Y: M \to \R^+$ be an arbitrary smooth function and $\varphi: M \times (0,1] \to \R^{+}_0$ a solution of the conjugate heat equation in a Ricci flow background,
\begin{equation}
  \label{conjugatehflow}
  \begin{split}
    \partial_t \varphi(\cdot,t) &= -\Delta_{\ell(t)} \varphi(\cdot,t)+\Sc_{\ell}(\cdot,t) \varphi(\cdot,t) \quad \text { on } M  \text { for all  } t \in (0,T], \\
    \varphi(\cdot,Y) &= \varphi_Y(\cdot).
  \end{split}
\end{equation} 
 
One sees using this equation, and the fact that $\partial_t \Sc_{\ell} = \lap_\ell \Sc_{\ell} + 2|\Rc_{\ell}|_{\ell}^2$, that
\begin{equation*}
  \partial_t \int_M (\Sc_{\ell(t)}(\cdot) + b) \phi(\cdot,t) d\ell(t) = \int_M 2 \phi |\Rc_{\ell(t)}|_{\ell(t)}^2 d\ell(t) \geq 0.
\end{equation*}
Hence, if $\lim_{t \downto 0} \int_M (\Sc_{\ell(t)}(\cdot)+b) \phi(\cdot,t) d\ell(t) \geq 0$, then $\int_M (\Sc_{\ell(Y)}(\cdot)+b) \phi(\cdot,Y) d\ell(Y) \geq 0$ and hence $\Sc_{\ell(Y)}(\cdot)+b \geq 0$ for all $Y \in (0,T)$ as $Y>0$ was arbitrary and $\varphi_Y: M \to \R^+$ was an arbitrary smooth function.

We consider metrics $g_0 \in W^{1,2+\si}$ which are $\ep(n)$-almost continuous and satisfy $\Sc_{g_0}(\cdot)+b \geq 0$ in the distributional sense, the definition of which we now motivate and present. 
The formula  for the  scalar  curvature of a smooth Riemannian metric $g$ can be written with the help of a smooth background metric as 
\begin{equation*}
  \Sc_{g} = \Sc_h - \div_h(Z(g,h)) + L(g,h),
\end{equation*}
where
\begin{align}
  \label{Ldefinition}
  L(g,h) &= \Sc_h - (\gradh_k g^{ij}) \curlT(g,h)^k_{ij} + (\gradh_k g^{ik}) \curlT(g,h)^j_{ji} \\
  \nonumber &\quad + g^{ij} \big(\curlT(g,h)^k_{kl} \curlT(g,h)^{l}_{ij} - \curlT (g,h)^k_{jl} \curlT(g,h)^l_{ik}\big),\\
  \nonumber \curlT(g,h)^i_{jk} &:= \frac{1}{2}g^{il}(\gradh_j g_{kl} + \gradh_k g_{jl} - \gradh_l g_{jk}),  
\end{align}
and
\begin{align}
  \nonumber Z^k &= g^{ij} \curlT(g,h)^k_{ij} - g^{ik} \curlT(g,h)^j_{ji}, \\
  \label{Zdefinition}
  Z(g,h) &= Z^{k} \cdot \frac{\partial}{\partial x^k}.
\end{align}
Testing this with a smooth function $\phi: M \to \R^+$, we see, after integrating by parts and commuting second order terms, that  
\begin{align*}
  &\int_M (\Sc_{g}+b) \phi dg \\
  &\quad = \int_M \big(\Sc_h - \div_h(Z(g,h)) +L(g,h)+b \big) \phi dh \\
  &\quad = \int_M \big( L(g,h) \left(\phi \frac{dg}{dh}\right) + h(Z(g,h),\gradh \left(\phi \frac{dg}{dh}\right)) + b \left(\phi \frac{dg}{dh}\right) \big) dh.
\end{align*} 
We sometimes write $L(g,h) = g^{-1} * g^{-1} * g^{-1} * \gradh g * \gradh g$ and $Z (g,h) = g^{-1} * g^{-1} * \gradh g$.

These considerations motivate the following definition (given in the paper of Lee--LeFloch~\cite{lee2015positive}, see also the earlier paper of LeFloch--Mardare~\cite{LeFlochMardare2007}, where various notions of distributional curvatures are presented, or the later paper of Sormani--Tian--Wang~\cite{SorTianWang}, for a version using notation similar to ours).
Here, $\frac{dg}{dh}:M \to \R^+$ is the function which is given locally by
$\frac{dg}{dh}(x)= \frac{\sqrt{\det G(x)}}{\sqrt{\det H(x)}}$, where $G(x)= (g_{ij}(x))_{i,j \in \{1,\ldots,n\}}$ and $H(x) = (h_{ij}(x))_{i,j \in \{1,\ldots,n\}}$.

\begin{defi}[Lee--LeFloch~\cite{lee2015positive}]
\label{LeeLeFloch} 
For a Riemannian metric $g \in W^{1,2}_{\loc}(M,h)$ with $\frac{1}{a} h \leq g \leq a h$ we say that \emph{$\Sc_{g}+ b \geq 0$ in the distributional sense} if 
\begin{equation*}
  \int_M \big( L(g,h) \left(\phi \frac{dg}{dh}\right) + h(Z(g,h),\gradh \left(\phi \frac{dg}{dh}\right)) + b \left(\phi \frac{dg}{dh}\right) \big) dh \geq 0
\end{equation*}
for all smooth non-negative functions $\phi: M \to \R^+$ with compact support in $M$, where $\frac{dg}{dh} = \frac{\sqrt{g}}{\sqrt{h}}$, and $L$ and $Z$ are defined in \eqref{Ldefinition} and \eqref{Zdefinition}.
\end{defi}

We will show for $g_0 \in W^{1,2+ 4\si}$ which are $\ep(n,\si)$-almost continuous and satisfy $\Sc_{g_0}(\cdot) + b \geq 0$ in the distributional sense that then a related Ricci flow solution $\ell$ does indeed satisfy $\lim_{t \downto 0} \int_M (\Sc_{\ell(t)}(\cdot)+b) \phi(\cdot,t) d\ell(t) \geq 0$, and hence $\Sc_{\ell(t)}(\cdot)+b \geq 0$ (and hence $\Sc_{g(t)}(\cdot)+b \geq 0$) in the strong sense for all $t>0$. In order to show that $\lim_{t \downto 0} \int_M (\Sc_{\ell(t)}(\cdot)+b) \phi(\cdot,t) d\ell(t) \geq 0$, we require estimates on $\phi$, $\grad \phi$, and on the solution $\ell(t)$ (respectively $g(t)$) itself.
 
We also require the following interpolation lemma, which can be shown in a straightforward manner using integration by parts, and is an $n$-dimensional Riemannian manifold version of a $1$-dimensional lemma in $\R$ of Maz'ja which appears in \cite{Mazja}, the proof being essentially the same. 
 
\begin{lem}[{cf. Lemma 1 in \cite[§ 8.2.1]{Mazja}}]
  \label{gradientlemma}
  Let $(M^n,g)$ be a smooth closed $n$-dimensional Riemannian manifold and let $f: M \to \R^+$ be a smooth function. Then we have
  \begin{equation*}
    \int_M \frac{|\gradgg  f|_{g}^4}{f^2} dg \leq c(n) \int_M |\gradgg^2 f|_{g}^2 dg.
  \end{equation*}
\end{lem}

\begin{proof}
In the following, $\gradg$ denotes $^{g}\nabla$, the covariant derivative with respect to $g$, and $|\cdot|_g$ denotes the norm of a tensor with respect to $g$.
Integrating the quantity $|\gradg f^{\frac{1}{2}}|_g^4$, we obtain
\begin{align*}
\int_M |\gradg f^{\frac{1}{2}}|_g^4 dg &= \int_M g^{kl} g^{rs} \gradg_k (f^{\frac{1}{2}}) \gradg_l (f^{\frac{1}{2}}) \gradg_r (f^{\frac{1}{2}}) \gradg_s (f^{\frac{1}{2}}) dg \\
&= - \int_M g^{kl} g^{rs} f^{\frac{1}{2}} \gradg_k \gradg_l(f^{\frac{1}{2}}) \gradg_r (f^{\frac{1}{2}}) \gradg_s (f^{\frac{1}{2}})dg \\
&\quad - 2 \int_M g^{kl} g^{rs} f^{\frac{1}{2}} \gradg_l(f^{\frac{1}{2}}) \gradg_k \gradg_r (f^{\frac{1}{2}}) \gradg_s (f^{\frac{1}{2}})dg \\
&= - \frac{1}{2} \int_M g^{kl} g^{rs} f^{\frac{1}{2}} \gradg_k \big( (\gradg_l f) \cdot f^{-\frac{1}{2}} \big) |\gradg (f^{\frac{1}{2}})|_g^2 dg \\
&\quad - \int_M g^{kl} g^{rs} f^{\frac{1}{2}} \gradg_l(f^{\frac{1}{2}}) \gradg_k \big( (\gradg_r f) \cdot f^{-\frac{1}{2}} \big) \gradg_s (f^{\frac{1}{2}}) dg \\
&= - \frac{1}{2} \int_M g^{kl} (\gradg_k \gradg_l f) |\gradg (f^{\frac{1}{2}})|_g^2 dg \\
&\quad + \frac{1}{4} \int_M f^{-1} g^{kl} (\gradg_l f)( \gradg_k f) |\gradg (f^{\frac{1}{2}})|_g^2 dg \\
&\quad - \int_M g^{kl} g^{rs} \gradg_l (f^{\frac{1}{2}}) \gradg_s (f^{\frac{1}{2}}) (\gradg_k \gradg_r f) dg \\
&\quad + \frac{1}{2} \int_M f^{-1} g^{rs}g^{kl} \gradg_l (f^{\frac{1}{2}}) (\gradg_k f) (\gradg_r f) \gradg_s (f^{\frac{1}{2}}) dg \\
&= \int_M \gradg^2 f * \gradg f^{\frac{1}{2}} * \gradg f^{\frac{1}{2}} dg + \int_M |\gradg (f^{\frac{1}{2}})|_g^4 + 2 \int_M |\gradg (f^{\frac{1}{2}}) |_g^4 dg, 
\end{align*}
and hence
\begin{align*}  
  \int_M |\gradg (f^{\frac{1}{2}})|_g^4 dg &= - \int_M \gradg^2 f * \gradg f^{\frac{1}{2}} * \gradg f^{\frac{1}{2}} dg \\
  &\leq  c(n) \int_M |\gradg^2 f|_g^2 dg + \frac{1}{2} \int_M |\gradg (f^{\frac{1}{2}})|_g^4 dg,
\end{align*}
which implies
\begin{align*} 
  \frac{1}{4} \int_M \frac{|\gradg f|_g^4}{f^2} dg &= \int_M |\gradg (f^{\frac{1}{2}})|_g^4 dg \\
  &\leq c(n)\int_M |\gradg^2 f|_g^2 dg,
\end{align*}
as claimed.
\end{proof}

For fixed $(y, s) \in M \times \left(0, T \right]$, the conjugate heat kernel in a Ricci flow background setting is the function $K(y, s; x, t)$, defined for $0 \leq t < s$ and $x \in M$ satisfying
\begin{equation}
  \label{conjugatehkernel}
  \begin{split}
    \left(-\partial_t-\Delta_{\ell(t)} + \Sc_{\ell(t)}(x) \right) K(y, s; x, t) &= 0, \\
    \lim _{t \upto s} K(y, s; x, t) &= \delta_y(x),
  \end{split}
\end{equation}
where $\delta_y$ is the Dirac measure supported on $\{y\}$. For fixed $(x,t) \in M \times (0,T)$, $K(x,t;\cdot,\cdot)$ also satisfies $\left(\partial_s-\Delta_{\ell(s)}\right) K(y, s; x, t) = 0$, where $\Delta_{\ell(s)}$ is the Laplacian with respect to $\ell(s)$. 
In Chapter~24 of \cite{chowetal},
it is shown that \eqref{conjugatehkernel} has a solution, and hence there is a solution $\varphi$ of \eqref{conjugatehflow} given by 
\begin{equation}
  \label{conj}
  \varphi(x,t) = \int_M K(y, Y; x, t) \varphi_Y(y) d \mu_{\ell(Y)}(y).
\end{equation}

In the following, we use the notation $f^{-}(\cdot) = \max(-f(\cdot),0)$ for any function $f: X \to \R$.

\begin{thm}\label{conjugatethm}
  Let $M^n$ be a closed manifold and $(M,\ell(t))_{t \in (0,T]} $ a smooth solution of the Ricci flow, with $\Sc_{\ell}(t) \geq -\frac{\ep}{t}$ for all $t \in (0,T]$ for some constant $\ep>0$.
  For $0<Y\leq T$,  and  any smooth function $\phi_Y: M \to \R_0^+$, let $\phi: M \times (0,Y] \to \R_0^+$ be the solution \eqref{conj} of the conjugate heat equation \eqref{conjugatehflow} with $\phi(\cdot,Y) = \phi_Y(\cdot)$.
  \begin{enumerate}[label=(\roman*)]
    \item We then have
      \begin{equation*}
        \phi(\cdot,t) \leq \frac{Y^{\ep} \sup_M \phi_Y}{t^{\ep}} \quad \text{ for all } t \in (0,Y].
      \end{equation*}
    \item Let $p \in [1,\infty)$, $\al  \in (1,\infty)$ and $r = \frac{\al}{\al-1}$ satisfy $\ep p r <1$. If we further assume that  
      \begin{align*} 
        L_{\al} &:= \int_0^{T} \int_M |{\Sc^{-}_{\ell}} |^{\al}(x,s) d\ell(x,s) \, ds < \infty, \\
        V &:= \sup_{t \in (0,T]} \vol(M,\ell(t)) < \infty,
      \end{align*}
      then we have
      \begin{equation*}
        \int_M \phi^p(t) d\ell(t) \leq c(p,\al,\ep,\phi_Y,V,L_{\al},Y) < \infty \quad \text{ for all } t \in (0,Y].
      \end{equation*}
  \end{enumerate}
\end{thm}

\begin{proof}
For $\tau \in [0,Y)$ we write $\psi: M \times [0,Y) \to \R^+_0$, $\psi(x,\tau): = \phi(x,Y-\tau)$ 
for all $x\in M$, $\tau \in [0,Y)$.  Then we see that 
\begin{equation*}
  \parttau \psi(x,\tau) = \lap_{\ell(Y-\tau)}\psi(x,\tau) - \Sc_{\ell(Y-\tau)}(x)\psi(x,\tau),
\end{equation*}
for all $x\in M$, $\tau \in [0,Y)$.  
In order to prove (i), as a first step we estimate 
\begin{equation*}
  \parttau \psi(x,\tau) \leq \lap_{\ell(Y-\tau)}\psi(x,\tau) + \frac{\ep}{Y-\tau} \psi(x,\tau).
\end{equation*}

Arguing as in the proof of the ODE lemma~C.1 in Lamm--Simon~\cite{lamm2023ricci} (see also the proof of Theorem~2.3 in Bamler--Chen~\cite{bamlerchen} or the proof of Theorem~2.9 in Burkhardt-Guim~\cite{burkhardt-guim2024smoothing}, where a similar ODE argument was used), we get
\begin{align*}
  \parttau (Y-\tau)^{\ep} \psi (x,\tau) &= (Y-\tau)^{\ep} \parttau  \psi(x,\tau) - \ep(Y-\tau)^{\ep-1} \psi(x,\tau)\\
  &\leq \lap_{\ell(Y-\tau)} \big((Y-\tau)^{\ep} \psi(x,\tau)\big) \\
  &\quad + \ep {(Y-\tau)}^{\ep-1} \psi(x,\tau) - \ep(Y-\tau)^{\ep-1} \psi(x,\tau) \\
  &\leq  \lap_{\ell(Y-\tau)} \big((Y-\tau)^{\ep} \psi(x,\tau)\big) 
\end{align*}
for all $x\in M$, $\tau \in [0,Y)$. 
Using the maximum principle, we see for $t= Y-\tau$ that
\begin{equation}
  \label{intermedpsi} 
  t^{\ep} \phi(x,t) = (Y-\tau)^{\ep} \psi( x,\tau) \leq Y^{\ep} \sup_M \psi(\cdot,0) = Y^{\ep} \sup_M \phi_Y,
\end{equation}
which means
\begin{equation*}
  \phi(x,t) \leq \frac{Y^{\ep} \sup_M \phi_Y}{t^{\ep}} 
\end{equation*}
for all $x \in M$, $t\in [0,Y)$.
This proves (i).

To prove the second estimate, we calculate
\begin{align*}
  &\parttau \int_M \psi^p(\tau) d\ell(Y-\tau) \\
  &\quad = \int_M \Big( \lap_{\ell(Y-\tau)} (\psi^p (\tau)) - p(p-1) \psi^{p-2}(\tau)|\grad_{\ell(Y-\tau)} \psi(\tau)|_{\ell(Y-\tau)}^2   \\
  &\qqqquad -(p - 1)\psi^p (\tau) \Sc_{\ell(Y-\tau)} \Big)d\ell(Y-\tau) \\
  &\quad \leq \int_M |(p- 1)\psi^p(\tau) \Sc^{-}_{\ell(Y-\tau)}| d\ell(Y-\tau) \\
  &\quad \leq \frac{c(p,\phi_Y,Y)}{(Y-\tau)^{p\ep}} \int_M |\Sc^{-}_{\ell(Y-\tau)}| d\ell(Y-\tau),
\end{align*}
in view of \eqref{intermedpsi}.    
Integrating in time from $\tau =0$ to $\tau =S$,   and using the Hölder inequality for $\al \in (1,\infty)$, $r := \frac{\al}{\al -1}$ (so that $\frac{1}{\al} + \frac{1}{r} =  1$), we obtain   
\begin{align*}
  &\int_M  \psi ^p(x,S) d\ell(Y-S)(x) \\
  &\quad \leq \int_M \psi ^p(x,0) d\ell(Y)(x) + \int_0^S \int_M \frac{c_0(p,\phi_Y,Y)}{(Y-\tau)^{p\ep}} |\Sc^{-}_{\ell(Y-\tau)}| d\ell(Y-\tau ) d\tau \\
  &\quad \leq \int_M \psi^p(x,0) d\ell(Y)(x) \\
  &\qquad + \Big(\int_0^S\int_M \frac{c_0^{r}(p,\phi_Y,Y)}{(Y-\tau)^{rp\ep}} d\ell(Y-\tau ) d\tau \Big)^{\frac{1}{r}} \\
  &\qqquad \cdot \Big(\int_0^S \int_M |\Sc^{-}_{\ell(Y-\tau)} |^{\al}  d\ell(Y-\tau) \, d\tau \Big)^{\frac{1}{\al}} \\
  & \quad \leq \int_M \psi^p(x,0) d\ell(Y)(x) + c_1(V,p,\ep,\phi_Y,Y,\al)(L_{\al})^{\frac{1}{\al}},
\end{align*}
as required. 
\end{proof}

With this information at hand, we are ready to prove further estimates for solutions of the Ricci--DeTurck flow (and the related Ricci flow solutions) which start from a $W^{1,2+\si}$-metric.
 
\begin{thm}
  For any $n\in \N$, $\si\in (0,\frac{1}{4})$ there exists a small $0<\ep_1(\si,n) \leq \ep_0(n)$, $\ep_0(n)$ from Theorem~\ref{W1siDeturckthm}, such that the following holds.
  Let $(M^n,h)$ be a smooth, connected, closed $n$-dimensional Riemannian manifold and let $g_0 \in W^{1,2+2\si}(M)$ be such that
  \begin{equation*}
    |g_0-h| \leq \ep_1(\si,n).
  \end{equation*}
  Assume further, for some $b \in \R$, that $\Sc_{g_0} + b \geq 0$ in the distributional sense (see Definition~\ref{LeeLeFloch}).

  Then the solution $g(t)_{t\in (0,T(n,K_0)]}$ of the Ricci--DeTurck flow from Theorem~\ref{C0RicciDeTurckthm} constructed in \cite{simon2002deformation} satisfies $\Sc_{g(t)} +b \geq 0$ (respectively $\Sc_{\ell(t)} + b \geq 0$ for any related Ricci flow solution $\ell(t)$ from Section~\ref{RDeTurckRicci}) for all $t \in (0,T)$ in the smooth sense.
\end{thm}

\begin{proof}
We scale the initial data once, if necessary, so that $K_0 + K_1 \leq 1$.
Let $0<Y<T$ and the smooth function $\phi_Y:M \to \R^+$ be given and $\phi: M \times (0,Y] \to \R^+$ be the solution \eqref{conj} to the conjugate heat flow equation \eqref{conjugatehflow} with background Ricci flow $\ell(t) := (\Phi(t))^{*}g(t)$, where $\Phi$ comes from \eqref{ODEDe} for some fixed $S \in (0,T)$.
As explained above, the solution exists: see Chapter~24 of \cite{Chow-Lu-Ni}.

We define $F(t) := \int_{M} \hat{\phi}^2(x,t)f(x,t) dh$, where $f(x,t) := (1+r(x,t))|\grad g(t)|^2$, $\hat \phi(x,t) := \phi((\Phi_t)^{-1}(x),t)$, and $r(x,t) := 1 + \La |g-h|^2(x,t)$.
Using $\hat \phi(\Phi(y,t),t)= \phi(y,t)$, and $\partt \Phi(y,t) = V( \Phi(y,t))$, where $V$ is the vector field defined in \eqref{defnV}, and the fact that $\ell(t) = \Phi_t^*g(t)$, we see that
\begin{align*}
\partt \hat \phi(x,t) = -\lap_{g(t)} \hat \phi (x,t) + \Sc_{g(t)}(x) \hat \phi(x,t) 
-h(\grad \hat \phi(x,t),V(x,t)).
\end{align*}
We use this freely in the following calculations. 
Also, we notice from  \eqref{eq:SobolevEstimate2} that 
\begin{align*}
&\int_{0}^T \int_{M}  |\Rm(g(t))|_{g(t)}^2 dg(t) \leq c(n) \int_{0}^T \int_{M}  |\Rm(g(t))|_{h}^2 dh \leq C(g_0)
\end{align*}
and hence, for $p(\si) = \frac{2+ \si}{\si}$,
\begin{align}
\int_{M}  {\phi}^{p(\si)}(x,t) d\ell(t) \leq C< \infty \label{littlephi}
\end{align} 
if $\ep_1(\si,n) \leq \min(\si^3,\ep_0(n))$, in view of (ii) in Theorem~\ref{conjugatethm} with the choice of $\al=2$, since $\ep_1(\si) p(\si) \cdot 2 \leq 2 \si^2(2+\si) <1$.

Defining
\begin{equation*}
  F(t) = \int_{M} (1+r(x,t)) |\grad g(t)|^2 \hat{\phi}^2(x,t) dh,
\end{equation*}
we estimate
\begin{align*}
  F(t) &\leq C \int_{M} |\grad g(t)|^{2+\si}dh + C \int_{M} \hat{\phi}^{p(\si)}(x,t) dh \\
  &\leq C +  C \int_{M} \hat{\phi}^{p(\si)}(x,t) dg(t) \\
  &= C + C \int_{M} {\phi}^{p(\si)}(x,t) d\ell(t) \\
  &\leq C  < \infty 
\end{align*}
for all $t \in (0,Y]$, where we used \eqref{bettergradest} and \eqref{littlephi}, and
$\frac{1}{2} h \leq g(t) \leq 2h$ for all $t\in (0,T)$.

We will use the following facts in the calculations that follow: $\lap_g z = g^{ij} \, \gradh^2_{ij} z + g^{-1}*g^{-1} * \gradh g * \grad z$, which holds for any smooth function $z$, and hence, integrating by parts, we see for any smooth functions $f$ and $z$ that
\begin{align*}
  &\int_M \big( z g^{ij} \, \gradh^2_{ij} f - f \lap_g z \big) dh \\
  &\quad = \int_M \big( z g^{ij} \, \gradh^2_{ij} f - f g^{ij} \, \gradh^2_{ij} z \big) dh + \int_M (f g^{-1}*g^{-1} * \gradh g * \grad z) dh \\
  &\quad = \int_M \big( z g^{ij} \, \gradh^2_{ij} f + g^{ij} \, \gradh_i f \gradh_{j} z +   f  g^{-1}* g^{-1} * \gradh g * \grad z \big) dh \\
  &\quad = \int_M \big( g^{-1} * g^{-1} * \gradh g * \gradh f * z   + f g^{-1}*  g^{-1} * \gradh g * \grad z \big) dh .
\end{align*}

We choose $z := \hat \phi^2$, and, as above, $f(x,t) := (1+r(x,t))|\grad g(t)|^2$, $r(x,t) := 1 + \La |g-h|^2(x,t)$ with $\Lambda = \frac{1}{\sqrt{\ep_0(n)}}$ in this estimate. We then use the above estimates and \eqref{finalomega} to estimate the time derivative of $F(t)$ as follows:  
\begin{align*}
  \partt F(t) &= \partt \int_M (1+r(x,t)) |\grad g(t)|^2 \hat{\phi}^2(x,t) dh \\
  &\leq \int_M \big(\hat \phi^2 g^{ij} \, \gradh^2_{ij} f - f \lap_{g(t)}\hat \phi^2 \big) dh + \int_M f |\gradgg \hat \phi|_g^2 dh \\
  &\quad + \int_M \hat \phi^2(x,t) \big( c(n) - |\gradh^2 g|^2 - \La|\gradh g|^4 \big) dh \\
  &\quad + \int_M 2 f \Sc_{g}\hat \phi^2 - f h(x)( \gradh \hat \phi^2 (x,t), V(x,t) ) dh \\
  &\leq \int_M \big( g^{-1} * g^{-1} * \gradh g * \gradh f * \hat \phi^2  + f \hat \phi g^{-1} * g^{-1} * \grad g * \grad \hat \phi \big) dh \\
  &\quad - \int_M \hat \phi^2 ( |\gradh^2 g|^2 + \La|\gradh g|^4) dh +c(n)\int_M f|\gradg \hat \phi|^2 dh \\
  &\quad + K(n) \int_M  f^2 |\hat{\phi}|^2 dh + \frac{1}{K(n)} \int_M \hat \phi^2 |\Rm(g(t))|^2 dh + c(n) \int_M \hat \phi^2 dh\\
  &\leq \int_M (\gradh^2 g * \gradh g + \gradh g * \gradh g * \gradh g) * (\gradh g * \hat{\phi}^2 + \hat \phi * \gradh \hat \phi) dh \\
  &\quad - \frac{1}{2}\int_M \hat \phi^2(x,t) ( |\gradh^2 g|^2 + \La |\gradh g|^4) dh + c(n)\int_M f |\grad \hat \phi |^2 dh \\
  &\quad + \frac{1}{K(n)} \int_M \hat \phi^2 |\Rm(g(t))|^2 dh + c(n) \int_M \hat \phi^2 dh \\
  &\leq \int_M c(n)f |\gradh \hat \phi|^2 dh + c(n) \int_M \hat \phi^2 dh  \\
  &\quad - \frac{1}{c(n)}\int_M \hat \phi^2(x,t) (|\gradh^2 g|^2 + \La |\gradh g|^4) dh \\
  &\leq \int_M \Big( \frac{100c^2(n)}{\La} \frac{|\gradh \hat \phi |^4 }{\hat \phi^2}  + \frac{\Lambda}{100c(n)}f^2 \hat \phi^2 \Big) dh \\
  &\quad + c(n) \int_M \hat \phi^2 dh - \frac{1}{c(n)}\int_M \hat \phi^2(x,t) (|\gradh^2g |^2 +\La |\gradh g|^4) dh, \\
  &\leq \int_M \frac{100c^2(n)}{\La} \frac{|\gradh \hat \phi |^4 }{\hat \phi^2} dh   + c(n) \int_M \hat \phi^2 dh \\
  &\quad - \frac{1}{2c(n)}\int_M \hat \phi^2(x,t) (|\gradh^2g |^2 + \La |\gradh g|^4) dh, \\
  &\leq \int_M \Big(\frac{100c^2(n)}{\La} \frac{|\gradgt \hat \phi|_{g(t)}^4 }{\hat \phi^2} - \frac{1}{2c(n)} \hat \phi^2(t) |\Riem(g(t))|_{g(t)}^2\Big) dg(t) \\
  &\quad + c(n) \int_M \hat \phi^2 dh,
\end{align*} 
where we used that 
\begin{equation*}
  |\Riem(g)|^2 \leq c(n) \big( |\gradh^2 g|^2 + |\gradh g|^4 \big),
\end{equation*}
as well as
\begin{equation*}
  \gradh^2 g * \gradh g * \gradh g \leq \frac{1}{P} |\gradh^2 g|^2 + P|\gradh g|^4
\end{equation*}
for $P>0$,  
\begin{equation*}
 \frac{1}{c(n)} |\gradh g|^4 \leq f^2 \leq c(n) |\gradh g|^4,
\end{equation*}
and $\frac{1}{2}h \leq g(t) \leq 2h$.
Hence 
\begin{align*}
  \partt F(t)
  &\leq \int_M \Big(\frac{100c^2(n)}{\La} \frac{|\gradgt \hat \phi |_{g(t)}^4 }{\hat \phi^2} - \frac{1}{2c(n)} \hat \phi^2(t) |\Riem(g(t))|_{g(t)}^2\Big) dg(t) +C(\phi_Y)\\
  &= \int_M \left(\frac{100c^2(n)}{\La} \frac{|\gradell \phi (t) |_{\ell(t)}^4}{\phi(t)^2} - \frac{1}{2 c(n)} \phi^2(t) |\Riem(\ell(t))|_{\ell(t)}^2 \right) d\ell(t)+C(\phi_Y) \\
  &\leq \int_M \left(\frac{c^3(n)}{\La} |\gradell^2 \phi(t) |_{\ell(t)}^2(t) - \frac{1}{2c(n)} \phi^2(t) |\Riem(\ell(t))|_{\ell(t)}^2 \right) d\ell(t)+C(\phi_Y),
\end{align*}
where we have used Lemma~\ref{gradientlemma}, and the fact that $\int_M |\hat \phi|^2dh$ is uniformly bounded (independent of time) in view of (ii) of Theorem~\ref{conjugatethm}.

We now consider $E(t) = \int_M |\gradell \phi(t)|_{\ell(t)}^2 d\ell(t)$.
From the paper of Jiang--Sheng--Zhang~\cite{jiang2023weak}, we see that
\begin{align*}
  \partt E(t) &\geq \int_M \Big( \frac{3}{2} |\gradell^2 \phi(t)|_{\ell(t)}^2(t) \\
  &\qqqquad - C(n) |\Riem(\ell(t))|_{\ell(t)} |\gradell\phi(t)|_{\ell(t)}^2(t) \\
  &\qqqquad - C(n)|\Riem(\ell(t))|_{\ell(t)}^2\phi^2(t) \Big) d\ell(t),
\end{align*}
and hence, using Lemma~\ref{gradientlemma}, we have   
\begin{equation*}
  \partt E(t) \geq \int_M \Big(|\gradell^2 \phi(t)|_{\ell(t)}^2 - C(n)|\Riem(\ell(t))|_{\ell(t)}^2\phi^2(t) \Big) d\ell(t).
\end{equation*}
Combining the two inequalities, we see 
\begin{equation*}
  \partt \big( E(t)-\sqrt{\La} F(t) +C(\phi_Y)t \big) \geq 0.
\end{equation*}

Hence, integrating in time we obtain
\begin{equation*}
  E(t) \leq E(Y) - \sqrt{\La} F(Y) + \sqrt{\La} F(t)  +C(\phi_Y)Y  \leq C_1(\phi_Y,\Lambda,n,Y)< \infty
\end{equation*}
for all $t\in (0,Y)$, since $F(t)$ is uniformly bounded, as explained above.
We have thus shown that
\begin{equation*}
  \int_M |\gradell \phi(t)|_{\ell(t)}^2 d\ell(t) \leq C_1(\phi_Y, n, \si)
\end{equation*}
for all $t\in (0,Y)$,
where $C_1(\phi, n, \si)$ is a constant which is independent of time. 
Therefore,
\begin{align*}
  \int_M |\gradh \hat \phi(t)|_{h}^2 dh
  &\leq c(n) \int_M |\gradgt \hat \phi(t)|_{g(t)}^2 dg(t) \\
  &= c(n) \int_M |\gradell \phi(t)|_{\ell(t)}^2 d\ell(t) \\
  &\leq C_1(\phi, n, \si),
\end{align*}
for all $t \in (0,Y)$,
where $C_1(\phi, n, \si)$ is a constant which is independent of time.

We are now in a position to show that $\lim_{t \downto 0} \int_M (\Sc_{\ell(t)}(\cdot)+b) \phi(\cdot,t) d\ell(t) \geq 0$. In the following, we write $L(g)$ and $Z(g)$ for $L(g,h)$ and $Z(g,h)$ coming from Definition~\ref{LeeLeFloch}, as $h$ is fixed.  We further assume in the following that $0<\ep_1(\si,n)$ satisifies $\ep_1(\si,n) \leq \min(\ep_0(n),\si^3)$: This guarantees, for $p(\si) = \frac{2(2+\si)}{\si}$, that $ \int_M |\hat \phi|^{p(\si)} dh \leq c(n, \phi_Y, \si) < \infty$ in view of (ii) of Theorem~\ref{conjugatethm} (with $\al = 2$), and $\int_M |\grad \hat \phi|^{2} dh < c(n, \phi_Y, \si) < \infty$, as we showed above.
That is,
\begin{align*}
\int_M \big( |\hat \phi|^{p(\si)} + |\grad \hat \phi|^{2} \big) dh \leq c(n, \phi_Y, \si) < \infty,
\end{align*}
which we use freely in the following. 

Assuming that the distributional scalar curvature at time zero is not less than $-b$, we use the estimates for $g(t)$ and $\hat \phi$ from above to obtain the following:
\begin{align}
  \nonumber &\int_M \Big(L(g_0) \left(\hat \phi(t) \frac{dg(t)}{dh}\right) + h(Z(g_0),\gradh \left(\hat \phi(t) \frac{dg(t)}{dh}\right)) + b \hat \phi(t) \frac{dg(t)}{dh} \Big) dh \\
  \nonumber &\quad = \int_M \Big(L(g_0) \left(\hat \phi(t) \frac{dg(t)}{dh}\right) + h(Z(g_0),\gradh \hat \phi(t) \frac{dg(t)}{dh} + \hat \phi(t) \gradh \left(\frac{dg(t)}{dh}\right)) \\
  \nonumber &\qqqquad + b \hat \phi(t) \left(\frac{dg(t)}{dh}\right)\Big) dh \\
  \nonumber &\quad = \int_M \Big( L(g_0) \hat \phi(t) \left(\frac{dg(t)}{dh} - \frac{dg_0}{dh}\right) + h(Z(g_0),\gradh \hat \phi(t) \left(\frac{dg(t)}{dh} - \frac{dg_0}{dh}\right)) \\
  \nonumber &\qqqquad + h(Z(g_0), \hat \phi(t) \gradh \left(\frac{dg(t)}{dh} - \frac{dg_0}{dh}\right)) + b \hat \phi(t) \left(\frac{dg(t)}{dh} - \frac{dg_0}{dh}\right)\Big) dh \\
  \nonumber &\qquad + \int_M \Big(L(g_0) \left(\hat \phi(t) \frac{dg_0}{dh}\right) + h(Z(g_0),\gradh \hat \phi(t) \left(\frac{dg_0}{dh}\right)) \\
  \nonumber &\qqqquad + \hat \phi(t) \gradh \left(\frac{dg_0}{dh}\right) + b \hat \phi(t) \left(\frac{dg_0}{dh}\right) \Big) dh \\
  \nonumber &\quad \geq \int_M \Big( (L(g_0) + b )\hat \phi(t) \left(\frac{dg(t)}{dh} - \frac{dg_0}{dh} \right) + h(Z(g_0),\gradh \hat \phi(t) \left(\frac{dg(t)}{dh} - \frac{dg_0}{dh}\right)) \\
  \nonumber &\qqqquad + h(Z(g_0), \hat \phi(t) \gradh \left(\frac{dg(t)}{dh} - \frac{dg_0}{dh}\right))   \Big) dh \\
  \nonumber & [\text{since the distributional scalar curvature of } g_0 \text{ is not less than } -b \text{ and } |\Sc_h| \leq c(n)] \\
  \nonumber &\quad \geq -|b|\left(\int_M \hat \phi(t)^2 dh \right)^{\frac{1}{2}} \left(\int_M \Big|\frac{dg(t)}{dh} - \frac{dg_0}{dh}\Big|^2  \right)^{\frac{1}{2}} \\
  \nonumber &\qquad - c(n) \left(\int_M |\gradh g_0|^{2+\si} dh\right)^{\frac{2}{2+\si}} \left(\int_M \hat \phi(t)^{p(\si)} dh\right)^{\frac{1}{p(\si)}} \\
  \nonumber &\qqquad \cdot \left(\int_M \Big|\frac{dg(t)}{dh} - \frac{dg_0}{dh}\Big|^{p(\si)} dh\right)^{\frac{1}{p(\si)}} \\
  \nonumber &\qquad - c(n) \left(\int_M |\gradh g_0|^{2} \left|\frac{dg(t)}{dh} - \frac{dg_0}{dh}\right|^2 dh \right)^{\frac{1}{2}} \left(\int_M |\gradh \hat \phi(t)|^{2} dh \right)^{\frac{1}{2}} \\
  \nonumber &\qquad - c(n) \left(\int_M |\gradh g_0|^{2+\si}dh \right)^{\frac{2}{2+\si}} \left(\int_M |\hat \phi(t)|^{  p(\si) } dh \right)^{\frac{1}{p(\si)}} \\
  \nonumber &\qqquad \cdot \left(\int_M \left|\gradh \left(\frac{dg(t)}{dh}-\frac{d g_0}{dh}\right)\right|^{2} dh\right)^{\frac{1}{2}} \\
  \nonumber &\quad \geq - c(n,b,h) \left(1+\int_M |\gradh g_0|^{2+\si} dh \right)^{\frac{2}{2+\si}} \left(\int_M \hat \phi(t)^{p(\si)} dh \right)^{\frac{1}{p(\si)}} \\
  \nonumber &\qqquad \cdot \left(\int_M \left|\frac{dg(t)}{dh} - \frac{dg_0}{dh}\right|^{p(\si)} dh\right)^{\frac{1}{p(\si)}} \\
  \nonumber &\qquad - c(n) \left(\int_M |\gradh g_0|^{2+\si} dh \right)^{\frac{1}{2+\si}} \left(\int_M \left|\frac{dg(t)}{dh} - \frac{dg_0}{dh} \right|^{p(\si)} dh \right)^{\frac{1}{p(\si)}} \\
  \nonumber &\qqquad \cdot \left(\int_M |\gradh \hat \phi(t)|^{2} dh \right)^{\frac{1}{2}} \\
  \nonumber &\qquad - c(n) \left(\int_M |\gradh g_0|^{2+\si} dh \right)^{\frac{2}{2+\si}} \left(\int_M |\hat \phi(t)|^{ p(\si) }dh \right)^{\frac{1}{p(\si)}} \\
  \nonumber &\qqquad \cdot \left(\int_M \left|\gradh \left(\frac{dg(t)}{dh}-\frac{d g_0}{dh}\right)\right|^{2} dh\right)^{\frac{1}{2}} \\
  &\quad =: \mathcal{E}(\phi_Y,n,\si,b,h,t) \to 0 \quad \text{ as } t \downto 0,
  \label{almostdistrib}
\end{align}
where we used 
\begin{align*}
  \int_M \left|\frac{dg(t)}{dh} - \frac{dg_0}{dh}\right|^{p} dh &\leq c(n)^p \int_M |g(t)-g_0|_h^{p} dh \\
  &\leq  c(n)^{2p-2} \int_M |g(t)-g_0|_h^{2} dh \\
  &\to 0
\end{align*}
as $t \downto 0$ for any $p \geq 2$, if $\frac{1}{2} h \leq g(t) \leq 2 h$, and 
\begin{align*}
  \gradh \left(\frac{dg}{dh}-\frac{d g_0}{dh}\right) 
  &= g^{-1}* \grad g  *\frac{dg }{dh}   - g_0^{-1}* \grad g_0 *\frac{dg_0}{dh}  \\
  &=  \left(g^{-1} -g_0^{-1}\right)* \grad g  *\frac{dg }{dh} \\
  &\quad  +g_0^{-1} *  \left(\grad g -\grad g_0\right)\frac{dg }{dh} \\
  &\quad  +g_0^{-1} * \grad g_0* \left(\frac{dg }{dh} -\frac{dg_0}{dh} \right)  
\end{align*}
to estimate 
\begin{align*}
  &\left(\int_M  \left|\gradh \left(\frac{dg(t)}{dh}-\frac{d g_0}{dh}\right)\right|^{2} dh\right)^{\frac{1}{2}}  \\
  &\leq  c(n) \left(\int_{M} |\grad g(t) |^{2+\sigma} +   |\grad g_0 |^{2+\si} dh \right)^{\frac{1}{2+\si}}\left(\int_M |g(t)-g_0|^{p(\sigma)} dh\right)^{\frac{1}{p(\sigma)}} \\
  &\quad +   c(n) \left(\int_{M} |\grad g(t)-\grad g_0|^2 dh\right)^{\frac{1}{2}} \\
  &\to 0
\end{align*}
as $t \downto 0$.

Recalling that the definition of the distributional scalar curvature is derived from integration by parts in the smooth setting, we see
\begin{align*}
  &\int_M (\Sc_{\ell(t)}(\cdot)+b) \phi(\cdot,t) d\ell(t) \\
  &\quad = \int_M (\Sc_{g(t)}(\cdot)+b) \hat \phi(\cdot,t) dg(t) \\
  &\quad = \int_M \Big( L(g(t)) \left(\hat \phi \frac{dg(t)}{dh}\right) + h(Z(g(t)),\gradh \left(\hat \phi \frac{dg(t)}{dh}\right) ) + b \left(\hat \phi \frac{dg(t)}{dh}\right) \Big) dh \\
  &\quad = \int_M \Big( L(g_0) \left(\hat \phi \frac{dg(t)}{dh}\right) + h(Z(g_0),\gradh \left(\hat \phi \frac{dg(t)}{dh}\right) ) + b \left(\hat \phi \frac{dg(t)}{dh}\right) \Big) dh \\
  &\qquad + \int_M \Big( (L(g(t))-L(g_0)) \left(\hat \phi \frac{dg(t)}{dh}\right) \\
  &\qqqqquad + h(Z(g(t)) - Z(g_0),\gradh \left(\hat \phi \frac{dg(t)}{dh}\right) ) \Big) dh \\
  &\quad \geq \mathcal{E}(t) \\
  &\qquad  - c(n)\left(\int_M |\grad g(t)- \grad g_0|^{2+\si} dh \right)^{\frac{2}{2+\si}} \left(\int_M \hat \phi^{p(\si)} dh \right)^{\frac{1}{p(\si)}} \\
  &\qqquad \cdot  \left( \int_M |\grad g(t)|^2 +|\grad g_0|^2 dh \right)^{\frac{1}{2}} \\
  &\qquad  -c(n) \left( \int_M |\grad g_0|^{2+\si}dh \right) \left( \int_M |g(t)-g_0|^{p(\si)}dh \right)^{\frac{1}{p(\si)}} \left( \int_M \hat \phi^{p(\si)}dh \right)^{\frac{1}{p(\si)}} \\
  &\qquad - c(n) \left(\int_M |\gradh \left(\hat \phi \frac{dg(t)}{dh}\right)|^2 \right)^{\frac{1}{2}} \left(\int_M |Z(g_0) -Z(g(t))|^2 dh \right)^{\frac{1}{2}} \\
  &\quad \to 0 \quad \text{as }  t \downto 0
\end{align*}
since $\int_M |\gradh g(t)-\gradh g_0|^{2+2\si} \to 0$ as $t \downto 0$, where we used \eqref{almostdistrib}, and 
\begin{align*}
  L(g) - L(g_0) & = \grad g * \grad g *g^{-1}*g^{-1}*g^{-1} - \grad g_0 * \grad g_0 *g_0^{-1}*g_0^{-1}*g_0^{-1} \\
  &=  (\grad g -\grad g_0)* \grad g *g^{-1}*g^{-1}*g^{-1} \\
  &\quad + \grad g_0 * (\grad g-\grad g_0) *g^{-1}*g^{-1}*g^{-1}\\
  &\quad + \grad g_0 * \grad g_0*( g^{-1}*g^{-1}*g^{-1} - g_0^{-1}*g_0^{-1}*g_0^{-1} ) 
\end{align*}
when estimating $\int_M \Big( (L(g(t))-L(g_0)) \left(\hat \phi \frac{dg(t)}{dh}\right)$ from below, and 
\begin{align*}
  Z(g) -Z(g_0) & = \grad g *g^{-1}*g^{-1} - \grad g_0 *g_0^{-1}*g_0^{-1}\\
  & = (\grad g -\grad g_0) *g^{-1}*g^{-1} + \grad g_0 *(g^{-1}*g^{-1} -g_0^{-1}*g_0^{-1})
\end{align*}
when estimating  $- \int_M |Z(g_0)-Z(g(t))|^2 dh$ from below, and 
\begin{align*}
  & \int_M |\gradh \left(\hat \phi \frac{dg(t)}{dh}\right)|^2dh\\
  &\quad \leq c(n) \int_M \left(|\gradh  \hat \phi(t)|^2 + |\hat \phi|^2 |\gradh g(t)|^2\right)dh \\
  &\quad \leq c(n) \int_M \gradh \hat \phi(t)|^2 dh + c(n) \left(\int_M |\grad g(t)|^{2+\si}\right)^{\frac{2}{2+\si}} \left(\int_M |\hat \phi|^{p(\si)} \right)^{2p(\si)} \\
  &\quad \leq c(\phi_Y,\si,n) < \infty
\end{align*}
for all $t\in (0,Y]$.
Using the argument presented at the beginning of this section, it follows that $\Sc_{g(t)} + b \geq 0$.
\end{proof}

\printbibliography

\end{document}